\documentclass[10pt]{amsart}
\usepackage[active]{srcltx}
\usepackage{amsmath,amsfonts,amssymb,amscd,graphicx,color,amsthm,enumerate,calc}
\usepackage[matrix,arrow,curve,cmtip]{xy} 

\numberwithin{equation}{section}

\def\OO{{\mathcal O}}

\def\al{\alpha}

\def\ga{\gamma}
\newcommand{\?}{\negthickspace}

\newcommand{\C}{{\ensuremath{\mathbb  C}}}
\newcommand{\Cbar}{\ensuremath{\overline{\mathbb  C}}}

\newcommand{\Chat}{{\ensuremath{\widehat{\mathbb  C}}}}

\newcommand{\Cstar}{\ensuremath{{{\mathbb  C}^*}}}

\renewcommand{\d}{d}

\newcommand{\dz}{\ensuremath{dz}}
\newcommand{\dzbar}{\ensuremath{d\overline{z}}}

\newcommand{\D}{\ensuremath{{\mathbb   D}}}
\newcommand{\Dbar}{\ensuremath{{\overline{\mathbb  D}}}}

\newcommand{\e}{\ensuremath{{\operatorname{e}}}}
\newcommand{\eps}{{\epsilon}}

\renewcommand{\H}{\ensuremath{\mathbb  H}}

\newcommand{\Hplus}{\ensuremath{{{\mathbb  H}_+}}}

\newcommand{\Id}{\ensuremath{{\operatorname{Id}}}}

\newcommand{\laz}{\ensuremath{{\la_0}}}

\newcommand{\mapfromto}[3]{\hbox{\ensuremath{#1 \colon #2 \longrightarrow #3}}}
\newcommand{\mate}{\perp \! \! \! \perp}

\newcommand{\N}{\ensuremath{\mathbb  N}}

\newcommand{\Q}{\ensuremath{\mathbb  Q}}

\newcommand{\R}{\ensuremath{\mathbb  R}}

\newcommand{\simla}{\ensuremath{\overset{\lambda}{\sim}}}
\newcommand{\simlaz}{\ensuremath{\overset{\lambda_0}{\sim}}}
\newcommand{\simlad}{\ensuremath{\overset{\lambda^d}{\sim}}}

\newcommand{\sm}{\ensuremath{{\smallsetminus}}}
\newcommand{\Sm}{\ensuremath{{\setminus}}}
\newcommand{\Sen}{\ensuremath{{{\mathbb  S}^1}}}
\newcommand{\Seninfty}{\ensuremath{{{\mathbb  S}^1_\infty}}}

\newcommand{\Sqt}{\ensuremath{{S/\negthickspace\sim}}}
\newcommand{\Sqte}{\ensuremath{{S_1/\negthickspace\sim_1}}}

\newcommand{\Sqtt}{\ensuremath{{S_2/\negthickspace\sim_2}}}

\newcommand{\Sto}{\ensuremath{{{\mathbb  S}^2}}}

\newcommand{\wbar}{\ensuremath{{\overline{w}}}}

\newcommand{\whR}{\ensuremath{\widehat R}}

\newcommand{\wtH}{\ensuremath{{\widetilde H}}}

\newcommand{\wtP}{\ensuremath{{\widetilde P}}}
\newcommand{\wtR}{\ensuremath{{\widetilde R}}}

\newcommand{\zbar}{\ensuremath{{\overline{z}}}}

\newcommand{\Z}{\ensuremath{\mathbb  Z}}
\def\al{\alpha}

\def\ga{\gamma}

\def\si{\sigma}

\def\eps{\epsilon}
\def\la{\lambda}
\def\La{\Lambda}

\def\R{\mbox{$\mathbb R$}}
\def\H{\mbox{$\mathbb H$}}
\def\C{\mbox{$\mathbb C$}}

\def\Z{\mbox{$\mathbb Z$}}
\def\Q{\mbox{$\mathbb Q$}}
\def\N{\mbox{$\mathbb N$}}

\def\D{\mbox{$\mathbb D$}}

\def\Bottcher{B{\"o}ttcher}

\def\Mobius{M{\"o}bius}

\theoremstyle{plain}
\newtheorem{newthm}{Theorem}

\newtheorem{theorem}{Theorem}[section]
\newtheorem{lemma}[theorem]{Lemma}
\newtheorem{proposition}[theorem]{Proposition}
\newtheorem{corollary}[theorem]{Corollary}
\newtheorem{conjecture}{Conjecture}
\newtheorem{defthm}[theorem]{Definition and Theorem}

\theoremstyle{definition}
\newtheorem{deflem}[theorem]{Definition and Lemma}
\newtheorem{definition}[theorem]{Definition}

\theoremstyle{remark}
\newtheorem{example}[theorem]{Example}
\newtheorem{remark}[theorem]{Remark}
\newtheorem*{claim1*}{Claim 1}
\newtheorem*{claim2*}{Claim 2}
\newtheorem*{claim3*}{Claim 3}
\newtheorem*{claim4*}{Claim 4}

\newcommand{\ALIGN}{\begin{align*}}
\newcommand{\ENDALIGN}{\end{align*}}
\newcommand{\ENUM}{\begin{enumerate}}
\newcommand{\ENUMa}{\begin{enumerate}[a.]}
\newcommand{\ENUMi}{\begin{enumerate}[i)]}
\newcommand{\ENDENUM}{\end{enumerate}}
\newcommand{\ITMZ}{\begin{itemize}}
\newcommand{\ENDITMZ}{\end{itemize}}
\newcommand{\EQN}[1] { \begin{equation}\label{#1} }
\newcommand{\ENDEQN}{\end{equation}}
\newcommand{\THM}{\begin{theorem}}
\newcommand{\REFEXA}[1] { \begin{example}\label{#1} }
\newcommand{\ENDEXA}{\end{example}}
\newcommand{\REM}{ \begin{remark}}
\newcommand{\ENDREM}{\end{remark}}
\newcommand{\REFTHM}[1] { \begin{theorem}\label{#1} }
\newcommand{\RTHM}[1] { \begin{theorem}[#1] }
\newcommand{\RREFTHM}[2] { \begin{theorem}[#1]\label{#2} }
\newcommand{\ENDTHM}{\end{theorem}}
\newcommand{\REFNTH}[1] { \begin{newthm}\label{#1} }
\newcommand{\ENDNTH}{\end{newthm}}
\newcommand{\REFPROP}[1]{\begin{proposition}\label{#1} }
\newcommand{\RREFPROP}[2]{\begin{proposition}[#1]\label{#2} }
\newcommand{\PROP}{\begin{proposition}}
\newcommand{\ENDPROP}{\end{proposition} }
\newcommand{\REFDEF}[1]{\begin{definition}\label{#1} }
\newcommand{\RREFDEF}[2]{\begin{definition}[#1]\label{#2} }
\newcommand{\DEF}{\begin{definition}}
\newcommand{\ENDDEF}{\end{definition} }
\newcommand{\REFLEM}[1]{\begin{lemma}\label{#1} }
\newcommand{\RREFLEM}[2]{\begin{lemma}[#1]\label{#2} }
\newcommand{\LEM}{\begin{lemma}}
\newcommand{\ENDLEM}{\end{lemma} }
\newcommand{\REFCOR}[1]{\begin{corollary}\label{#1} }
\newcommand{\COR}{\begin{corollary}}
\newcommand{\ENDCOR}{\end{corollary}}
\newcommand{\REFCONJ}[1]{\begin{conjecture}\label{#1} }
\newcommand{\CONJ}{\begin{conjecture}}
\newcommand{\ENDCONJ}{\end{conjecture}}
\newcommand{\RMRK}{\begin{remark}}
\newcommand{\ENDRMRK}{\end{remark}}
\newcommand{\REFDEFTHM}[1] { \begin{defthm}\label{#1} }
\newcommand{\RREFDEFTHM}[2] { \begin{defthm}[#1]\label{#2} }
\newcommand{\ENDDEFTHM}{\end{defthm}}

\newcommand{\corref}[1]{Corollary~\ref{#1}}
\newcommand{\defref}[1]{Definition~\ref{#1}}

\newcommand{\lemref}[1]{Lemma~\ref{#1}}
\newcommand{\thmref}[1]{Theorem~\ref{#1}}
\newcommand{\propref}[1]{Proposition~\ref{#1}}
\newcommand{\itemiref}[1]{(\ref{#1})}
\newcommand{\PROOF}{\begin{proof}}
\newcommand{\ENDPROOF}{\end{proof}}

\providecommand{\defn}[1]{\emph{#1}}
 
\newcounter{mylistnum}

\providecommand{\abs}[1]{\lvert#1\rvert}

\newcommand{\simhat}{\widehat{\sim}}

\newcommand{\wt}{{\tt w}}
\newcommand{\bt}{{\tt b}}

\def\mate{\hskip 1pt \bot \hskip -5pt \bot \hskip 1pt}

\newcommand{\ray}{{\operatorname{ray}}}
\newcommand{\simr}{\sim_\ray}

\title{On The Notions of Mating}

\author{Carsten Lunde Petersen}
\address{Carsten Lunde Petersen
  Institut for Natur, Systemer og Modeller
Bygn 27.2\\
Roskilde Universitet\\
Universitetsvej 1\\
DK-4000 Roskilde\\
Denmark} 

\author{Daniel Meyer}
\address{Daniel Meyer\\
Jacobs University Bremen\\ 
Campus Ring 1\\
28759 Bremen\\
Germany}

\date{\today}
\begin{document}

\maketitle

\noindent{Abstract:~}
The different notions of matings of pairs of equal degree polynomials 
are introduced and are related to each other as well as known results on matings.
The possible obstructions to matings are identified and related. 
Moreover the relations between the polynomials and their matings are discussed and proved. 
Finally holomorphic motion properties of slow-mating are proved.
\\
\\
\noindent{Resum{\'e}:~}
Les diff\'{e}rentes notions de accouplements de paires de polyn{\^o}mes 
de m{\^e}mes degr{\'e} sont introduits et sont reli{\'e}es les uns aux autres 
ainsi que les r{\'e}sultats connus concernant les accouplements.
Les obstructions possibles {\`a} les accouplements sont identifi{\'e}s et li{\'e}s.
En plus, les relations entre les polyn{\^o}mes et leurs accouplements 
sont discut{\'e}es et prouv{\'e}s.
Enfin des propri{\'e}t{\'e}s de mouvement holomorphes de l'accouplement lente 
sont prouv{\'e}s.

\section{Introduction}
The notion of mating of two polynomials of the same degree $d>1$ 
was introduced by Douady and Hubbard in order to understand and classify 
dynamically certain rational maps. Since its introduction and the first proven results by 
Tan Lei, Mary Rees and Mitsuhiro Shishikura, 
several more or less equivalent notions of matings have been introduced and in use.
This paper aims at introducing the different notions of matings and their applications.
We will discuss along the way the different issues which naturally arise in connection 
with the different definitions of matings. 

Fundamentally there are two different views on mating. 
\begin{itemize}
  \item 
  The constructive approach: 
  Mating is a procedure to construct new rational maps by combining two polynomials.
  \item 
  The descriptive approach: Mating is a way to understand the dynamics 
  of certain rational maps in terms of pairs of polynomials.
  \end{itemize}  
We shall pursue both views. In order to set the scene for the discussion properly 
let us start with reviewing the relevant background definitions and theorems. 
The reader novel to holomorphic dynamics will find full details and enlarged discussions in any 
of the monographs \cite{Milnor:ComplDyn}, \cite{CarlesonandGamelin}.

\section{Equivalence relations}
\label{sec:clos-equiv-relat}

In the construction of the mating we take the quotient of a highly
non-trivial equivalence relation. 
We shall thus take our starting point in a discussion of equivalence relations 
and the properties of quotients by equivalence relations.

\subsection{The lattice of equivalence relations}
\label{sec:latt-equiv-relat}

Let $R \subset S\times S$ be a relation on the set $S$. 
We shall follow the usual custom of writing $xRy$, $x\sim y$ or $x\sim_R y$ for $(x,y)\in R$.
The three forms  will be used interchangeably. 
Relations on $S$ are naturally \defn{partially ordered} by inclusion, 
i.e.,  the relation $R$ is \defn{bigger} than the
relation $Q$ if $Q\subset R$ or
$$
\forall\;x,y\in S: \qquad  xQy \Rightarrow xRy,
$$

The set of all relations on $S$ forms a \defn{lattice}
when equipped with this partial ordering. This means that
\defn{join} and \defn{meet} are well defined. Recall that
the join $Q \vee R=Q\cup R$ is the smallest relation bigger  
than $Q$ and $R$. Similarly the \defn{meet} $Q\wedge R= Q\cap R$ is
the biggest relation smaller than $Q$ and $R$. Similarly the join and
meet of an arbitrary family of relations is defined. Note that the meet
of a family of equivalence relations is again an equivalence
relation. 

Let $R$ be a relation on $S$. 
The \defn{equivalence relation generated} by $R$ 
is the smallest equivalence relation bigger than
$R$, i.e., the meet of all equivalence relations bigger than $R$. 
 
\subsection{Closed equivalence relations}
\label{sec:clos-equiv-relat-1}

The set of equivalence classes of an equivalence relation $R$ on a set $S$
forms a \emph{decomposition}, i.e., a family of pairwise disjoint
subsets of $S$ whose union is $S$. Clearly each decomposition of $S$
induces an equivalence relation on $S$ (two points of $S$ are
equivalent if and only if they are contained in the same set of the
decomposition). 
We denote by $[x]$ or more detailed $[x]_R$ the $R$-equivalence class of $x$. 
We denote by $\Sqt = S/R = \{[x] \mid  x\in S\}$
the space of equivalence classes or quotient space and by 
{\mapfromto {\Pi=\Pi_R} S \Sqt} the natural projection $\Pi(x) = [x]$.
When $S$ is a topological space we shall always assume that the quotient space 
$\Sqt$ is equipped with the quotient topology. 
That is a subset $U\subset \Sqt$ is open if and only if $\Pi^{-1}(U)$ is open in $S$.

Instead of quotient spaces one talks in geometric topology about
\emph{decomposition spaces}. The standard reference is Daverman's book
\cite{MR872468}. We will however stick to talking about equivalence
relations, instead of decompositions.

\begin{definition}[saturation]
  \label{def:saturation}
  Let $\sim$ be an equivalence relation on a set $S$. A set $U\subset
  S$ is called \emph{saturated} if $x\in U, y\sim x \Rightarrow y\in
  U$, equivalently $U= \Pi^{-1}(\Pi(U))$, equivalently $U$ is a union
  of equivalence classes. 

  The \emph{saturated interior} of a set $V\subset S$ is the set
  \begin{equation*}
    V^*:= \bigcup\{ [x] \mid [x]\subset V\},
  \end{equation*}
  i.e., the biggest saturated set contained in $V$. Note that $V^*$
  may be empty, even if $V$ is non empty.

  The \emph{saturation} of a set $A\subset S$ is the set
  \begin{equation*}
    A^\dagger:= \bigcup\{[x] \mid x\in A\} = \Pi^{-1}(\Pi(A)),
  \end{equation*}
  i.e., the smallest saturated set containing $A$. 
\end{definition}

Recall that a topological space $S$ is Hausdorff if and only if every two distinct points 
have disjoint neighborhoods or equivalently if and only if the diagonal 
$\Delta=\Delta_S =\{(x,x) \mid x\in S\}\subset S\times S$ is closed. 
Here as elsewhere we suppose the later equipped with the product topology.

\smallskip
In general there is very little that can be said about the quotient
space. However the standard assumption is that the equivalence
relation is \emph{closed}, the importance of which is shown by the
following.

\begin{deflem}
  \label{def:closed_eq}
  Let $S$ be a compact metric space. An equivalence relation $\sim$ on $S$ 
  is \defn{closed} if each $[x]$ is compact and one (hence all) of
  the following equivalent conditions is satisfied.
  
  \begin{enumerate}
  \item 
    \label{item:closed_closed}
    The set $\{(s,t) \mid s\sim t \} \subset S\times S$ is closed.
  \item 
    \label{item:closed_seqclosed}
    Let $(s_n)_{n\in \N}, (t_n)_{n\in \N}$ be convergent
    sequences in $S$. Then
    \begin{equation*}
      s_n\sim t_n
      \text{ for all } n\in \N,  \text{ implies } \lim s_n\sim \lim t_n. 
    \end{equation*}
  \item 
    \label{item:closed_Hausdorff}
    For any compact subset $C\subseteq S$: If $C$ is the Hausdorff limit 
    of a sequence of equivalence classes $([x_n])_{n\in \N}$,
    then $C\subseteq [x]$ for some $x\in S$. 
  \item 
    \label{item:def_usc}
    For any equivalence class $[x]$ and any neighborhood $U$ of $[x]$ there
    is a neighborhood $V\subset U$ of $[x]$, s.t.
    \begin{equation*}
      [y]\cap V\neq \emptyset \Rightarrow [y] \subset U.
    \end{equation*}
  \item  
    \label{item:closed_saturnbh}
    Each neighborhood $U$ of any equivalence class $[x]$ contains
    a \emph{saturated} neighborhood $V$ of $[x]$. 
    
  \item 
    \label{item:closed_saturint}    
    For each open set $U$ the set saturated interior $U^*$ 
    is open. 
    \setcounter{mylistnum}{\value{enumi}}
  \item 
    \label{item:closed_qtHausdorff}
    The quotient topology on $\Sqt$ is Hausdorff.
  \item
    \label{item:closed_saturationclosed}
    For any closed set $K\subset S$ the saturation
    \begin{equation*}
      K^\dagger:= \bigcup \{[x] \mid x\in K\}
    \end{equation*}
    is closed.
  \item 
    \label{item:closed_proj}
    The quotient map $\Pi\colon S \to \Sqt$ is closed.
  \item 
    \label{item:closed_qtMetrz}
    The quotient topology on $\Sqt$ is metrizable.  
  \end{enumerate}
\end{deflem}  
 
\begin{proof}
A round robin style proof of the equivalence of \itemiref{item:closed_closed} to 
\itemiref{item:closed_qtHausdorff} is: \\
\itemiref{item:closed_closed} $\Rightarrow$ 
\itemiref{item:closed_seqclosed} is a general property of closed sets.\\ 
\itemiref{item:closed_seqclosed} $\Rightarrow$ \itemiref{item:closed_Hausdorff}, 
because any two points $z,w\in C$ are limits of sequences $(z_n), (w_n)$ with 
$z_n, w_n\in [x_n]$ for all $n$.\\ 
\itemiref{item:closed_Hausdorff} $\Rightarrow$ \itemiref{item:def_usc}, 
by contra position, if not then there exists an equivalence class $[x]$,  
a neighborhood $U$ of $[x]$, which we can suppose to be open, a sequence 
of equivalence classes $[x_n]$ and points $z_n, w_n \in [x_n]$ such that 
$z_n$ converges to $z\in [x]$ and $S\Sm U \ni w_n$. 
Passing to a subsequence if necessary we can suppose $[x_n]$ 
converges to a compact set $C$, which by construction intersects both $[x]$ 
and the compact complement $S\Sm U$ and thus is not a subset of any equivalence class.\\
\itemiref{item:def_usc} $\Rightarrow$ \itemiref{item:closed_saturnbh}, 
let $[x]$, $U$ and $V$ be as in \itemiref{item:def_usc}, then the saturation $V^\dagger$ 
fulfills the requirements.\\ 
\itemiref{item:closed_saturnbh} $\Rightarrow$ \itemiref{item:closed_saturint} 
follows immediately.\\ 
\itemiref{item:closed_saturint} $\Rightarrow$ \itemiref{item:closed_qtHausdorff}. 
Let $[x]$ and $[y]$ be distinct and hence disjoint equivalence classes. 
Let $U([x])$ and $U([y])$ be corresponding disjoint open neighborhoods in $S$. 
Then $\Pi(U^*([x]))$ and $\Pi(U^*([y]))$ are disjoint neighborhoods of 
$\Pi(x)$ and $\Pi(y)$.\\
\itemiref{item:closed_qtHausdorff} $\Rightarrow$ \itemiref{item:closed_closed}
Let {\mapfromto {\Phi = \Pi\times\Pi} {S\times S}{\Sqt\times \Sqt}} 
i.e.,~$\Phi(x,y) = (\Pi(x),\Pi(y))$,
then $R=\Phi^{-1}(\Delta_{S/\sim})$ is closed, because $\Phi$ is continuous 
and $\Sqt$ is Hausdorff.\\ 
\itemiref{item:closed_saturint} $\Leftrightarrow$ \itemiref{item:closed_saturationclosed} 
By definition, if $U = S\Sm K$ then $U^* = S\Sm K^\dagger$.\\
\itemiref{item:closed_saturationclosed} $\Leftrightarrow$ \itemiref{item:closed_proj} 
by definition of the quotient topology, as $\Pi^{-1}(\Pi(K)) = K^\dagger$.\\
Finally trivially \itemiref{item:closed_qtMetrz} $\Rightarrow$ \itemiref{item:closed_qtHausdorff} 
whereas\\ \itemiref{item:closed_qtHausdorff} and its equivalent \itemiref{item:closed_saturint} 
$\Rightarrow$ \itemiref{item:closed_qtMetrz}, because a compact Hausdorff space is 
metrizable if and only if it is second countable, so that we just need to exhibit a countable 
basis for the quotient topology on $\Sqt$. For $\eps>0$ and $x\in S$ write 
$U([x],\eps) = \{x'\mid d(x',[x]) < \eps\}$. 
Then for each fixed $m\in\N$ the open saturated sets 
$U^*([x],1/m)\subseteq U([x],1/m)$ form a covering of the compact space $S$. 
Hence we may extract a finite sub-covering 
$V_{m,1}= U^*([x_{m,1}],m),\ldots, V_{m,n_m} =U^*([x_{m,n_m}],m)$. 
The projected sets $\Pi(V_{m,i})$ form a countable basis for the quotient topology 
on $\Sqt$ : An open neighborhood of a point $\Pi(x)$ 
is the projection of a saturated neighborhood $V$ of $[x]$. 
We shall thus find $m$ and $i$ such that $V_{m,i}\subset V$. 
Let $0< \delta = \inf\{\d(x',[x]) \mid \;x'\in S\Sm V\}$. 
And let $0< \eps = \inf\{\d(x',[x]) \mid \;x'\in S\Sm U^*([x],\delta/2)\}$. 
For $m$ with $1<m\eps$ choose $i$ such that $[x]\subset V_{m,i}$. 
Then there exists $y\in [x_{m,i}]$ with $\d(y,[x])< 1/m<\eps$ and hence 
$[x_{m,i}]\subset U^*([x],\delta/2)\subseteq U([x],\delta/2)$. 
Thus $V_{m,i}\subset U([x],\delta) \subseteq V$, as $1/m<\eps\leq\delta/2$.
\end{proof}

Property (\ref{item:closed_Hausdorff}) shows that in the case when
$\sim$ is closed ``small equivalence classes can converge to bigger
equivalence classes'' (but not vice versa!). This is the reason that
closed equivalence relations are also called \emph{upper
  semi-continuous}. The standard definition for upper semi-continuity
(in general topological spaces) is (\ref{item:def_usc}). 

\smallskip
Let $\varphi\colon S\to S'$ be a map. Then $\varphi$ \emph{induces} an
equivalence relation $\sim~=~\sim_\varphi$ on $S$: 
\begin{equation*}
  s\sim t \quad \text{if and only if} \quad \varphi(s)=\varphi(t),
\end{equation*}
for all $s,t\in S$. That is $\sim~= (\varphi\times\varphi)^{-1}(\Delta_{S'})$. 
Evidently the induced map {\mapfromto \psi {\Sqt} {\varphi(S)}} given by 
$\psi([x]) = \varphi(x)$ is a bijection. 
Moreover if $S, S'$ are topological spaces and $\varphi$ is continuous, 
then $\psi$ is continuous.

The following is a standard topological fact, for which the reader
shall easily provide a proof.
\begin{lemma}
  \label{lem:sim_induced}
  Let $\varphi\colon S\to S'$ be a continuous surjective map between
  topological spaces and $\sim$ be the equivalence relation on $S$
  induced by $\varphi$.  
  Then the induced map {\mapfromto \psi {\Sqt}{S'}} is a homeomorphism, 
  if and only if $\varphi(U)$ is open for every saturated open set.
  In particular $\psi$ is a homeomorphism if $S$ is compact and $S'$
  is Hausdorff. In this case $\sim$ is furthermore closed.  
\end{lemma}

\begin{lemma}[Closure of equivalence relation]
  \label{lem:usc_closure}
  Let $\sim$ be an equivalence relation on a compact metric space
  $S$. Then there is a unique smallest closed equivalence relation $\simhat$
  bigger than $\sim$. We call $\simhat$ the \defn{closure} 
  of $\sim$. 
\end{lemma}

\begin{proof}
  Let $(R_j)_{j\in J}$ be the family of all closed equivalence
  relations bigger than $\sim$. Then we define $\simhat$ as the meet
  (i.e., the intersection) of all $R_j$. Clearly $\simhat$ is a closed
  equivalence relation, bigger than $\sim$, and the smallest such
  relation. Uniqueness is evident as well. 
\end{proof}

Note that $\{(s,t)\mid s\;\simhat\; t\}$ is generally \emph{not} the
closure of 
$\{(s,t)\mid s\sim t\}$, which may fail to be transitive. An explicit
description of $\simhat$ is given in the proof of
Proposition~\ref{prop:clos_dynamics}. 

\subsection{Equivalence relations and (semi-)conjugacies}
\label{sec:equiv-relat-dynam}

Let $S_1$ and $S_2$ be spaces equipped with equivalence relations $\sim_1$ and $\sim_2$ 
respectively. 
A mapping {\mapfromto f {S_1} {S_2}} is called a \emph{semi-conjugacy 
for the equivalence relations} $\sim_1$ and $\sim_2$ if $x\sim_1 y\Rightarrow
f(x)\sim_2 f(y)$ for all $x,y\in S_1$. Equivalently $f^{-1}(U)$ is
$\sim_1$-saturated for any $\sim_2$-saturated set $U\subset S_2$. 
The map $f$ descends to (or induces) a map {\mapfromto F \Sqte \Sqtt} 
between the quotients $\Sqte$ and $\Sqtt$ given by 
\begin{equation*}
  F ([x]_1) := [f(x)]_2,
\end{equation*}
for all $x\in S_1$, if and only if $f$ semi-conjugates $\sim_1$ to $\sim_2$.

If $S_1=S_2=S$ and $\sim_1\;=\;\sim_2\;=\;\sim$ 
we also say that $\sim$ is \defn{$f$-invariant} 
and write $f/\?\sim$ for the quotient map $F$ above.

\begin{lemma}
  \label{lem:sim_dyn}
  Suppose {\mapfromto f {S_1} {S_2}} is continuous and that $f$ semi-conjugates 
  $\sim_1$ to $\sim_2$. Then {\mapfromto {F:=f/\?\sim\;}{\Sqte}{\Sqtt}} is continuous. 
\end{lemma}

\begin{proof}
  Let $[U']\subset \Sqtt$ be open. Then $U:=\Pi_2^{-1}([U'])\subset S_2$
  is open and $\sim_2$ saturated. 
  Thus $f^{-1}(U)\subset S_1$ is open and $\sim_1$ saturated so that
  $\Pi_1(f^{-1}(U))= F^{-1}([U'])\subset \Sqte$ is open.  
\end{proof}

\begin{lemma}
  \label{lem:sim_rel_finv}
  Let $R$ be a relation on a set $S$ that is invariant with respect to
  $f\colon S\to S$, i.e., $xRy \Rightarrow f(x)Rf(y)$ for all $x,y\in
  S$. Then the equivalence relation $\sim$ generated by $R$ is
  $f$-invariant. 
\end{lemma}
The proof is left as an easy exercise. 

\smallskip
Let $f\colon S\to S$, $g\colon S'\to S'$ be self-maps of the sets
$S,S'$ respectively, which we consider as dynamical systems. A
\emph{semi-conjugacy} from $f$ to $g$ is a surjection $\varphi \colon S\to S'$
such that $\varphi\circ f= g\circ \varphi$ on $S$, i.e., the following
diagram commutes

\begin{equation*}
  \xymatrix{
    S \ar[d]_{\varphi} \ar[r]^f
    &
    S \ar[d]^{\varphi} 
    \\
    S' \ar[r]^g
    & S'.
    }
\end{equation*}
We say $f$ is semi-conjugate to $g$ by $\varphi$. 
The map $g\colon S'\to S'$ is called a \emph{factor} 
of the dynamical system $f\colon S\to S$. 
If $S,S'$ are topological spaces and $\varphi$ is continuous,
then $\varphi$ is called a \emph{topological semi-conjugacy}.
If $\varphi$ is a
homeomorphism $\varphi$ is called a \emph{topological
  conjugacy}. Similarly if $\varphi$ is (quasi-) conformal, we call
$\varphi$ a (quasi-) conformal conjugacy. 

\begin{lemma}
  \label{lem:semi_sim}
   Let $f$ be semi-conjugate to $g$ by $\varphi$ as above. Then the
   following holds.
  \begin{itemize}
  \item  
    The equivalence relation $\sim$ induced by $\varphi$ is $f$-invariant.
  \item
    Assume $\varphi$ is a topological semi-conjugacy, $S$ compact, and
    $S'$ is Hausdorff.  
    Then 
    $f/\?\sim\; \colon \Sqt\, \to \Sqt$ is continuous and
    topologically conjugate to $g\colon S'\to S'$. 
  \end{itemize}
\end{lemma}

\begin{proof}
  The first statement follows immediately from the
  definitions. Indeed, let $x\sim y$ for some $x,y\in S$. Then
  $\varphi(x) =\varphi(y)$. Thus $\varphi\circ f(x)=g\circ
  \varphi(x)=g\circ \varphi(y)= \varphi\circ f(y)$, which means that
  $f(x)\sim f(y)$.  

  To see the second statement, note that $S/\!\sim$ is homeomorphic to
  $S'$ by Lemma~\ref{lem:sim_induced}, $f/\!\sim$ is continuous by
  Lemma~\ref{lem:sim_dyn}. It is easily verified that the
  homeomorphism $\psi\colon S/\!\sim \to S'$ satisfies $g\circ \psi =
  \psi \circ f/\!\sim$. 
\end{proof}

\begin{proposition}
  \label{prop:clos_dynamics}
  Let $S$ be a compact metric space, $f\colon S\to S$ be a continuous
  map, and $\sim$ be an $f$-invariant equivalence relation. Then the
  closure $\simhat$ of $\sim$ is $f$-invariant. 
\end{proposition}

\begin{proof}
  Let $f\colon S\to S$ be continuous.

  We first give an explicit description of the closure $\simhat$. We
  will use ordinal numbers, i.e., transfinite induction.

  Let $\sim$ be any equivalence relation on $S$. Here we think of
  $\sim$ as a subset of $S\times S$. Let $\eqsim$ be
  the closure of $\sim$ in $S\times S$. Note that $\eqsim$ is not
  necessarily an equivalence relation. 

  Now we define $\sim'$ as the equivalence relation generated by
  $\eqsim$. Note that this equivalence relation is not necessarily
  closed. 

  \begin{claim1*}
    If $\sim$ is $f$-invariant then $\sim'$ is $f$-invariant as well.
  \end{claim1*}

  \begin{proof}[Proof of Claim~1]
    We first note that $\eqsim$ is given by the following. For all
    $x,y\in S$ it holds $x \eqsim y$ if and only if there are sequences
    $x_n\to x$, $y_n\to y$ in $S$, such that $x_n \sim y_n$ for all
    $n\in N$. Since $\eqsim$ is closed in $S\times S$, as well as bigger
    than $\sim$ (which is $f$-invariant), it follows that
    \begin{equation*}
      f(x) = \lim f(x_n) \eqsim \lim f(y_n) = f(y).
    \end{equation*}
    Thus $\eqsim$ is $f$-invariant. 
    
    \smallskip
    Note that $\sim'$ is given by the following. For all $x,y\in S$ it
    holds $x\sim' y$ if and only if there is a finite sequence $x_0=x,
    x_1, \dots, x_n=y$ such that
    \begin{equation*}
      x_0 \eqsim x_1 \eqsim \dots x_{n-1}\eqsim x_n.
    \end{equation*}
    Since $\eqsim$ is $f$-invariant it follows that $\sim'$ is
    $f$-invariant. This finishes the proof of Claim~1. 
  \end{proof}
  
  We now define $\sim_\alpha$ for any ordinal number $\alpha$ as
  follows. Let $\sim_0\,:= \,\sim$ be the equivalence relation from the
  statement of the proposition. Assume $\sim_\alpha$ has been defined
  for all ordinals $\beta< \alpha$. If $\alpha$ is not a limit
  ordinal, i.e., $\alpha= \beta+1$ (for an ordinal $\beta$), then
  \begin{equation*}
    \sim_\alpha\,:= \,{\sim_\beta}'.
  \end{equation*}
  If $\alpha$ is a limit ordinal we define
  \begin{equation*}
    \sim_\alpha\,:= \bigvee_{\beta< \alpha} \sim_\beta,
  \end{equation*}
  i.e., the join of all $\sim_\beta$ with $\beta<\alpha$. Recall that
  this is the smallest equivalence relation bigger than all
  $\sim_\beta$ (for $\beta< \alpha$). 
  Note that
  $(\sim_\alpha)$ is an increasing sequence of equivalence relations,
  i.e., $\beta\leq \alpha \Rightarrow \;\sim_\beta\, \leq\,
  \sim_\alpha$. Thus in the case that $\alpha$ is a limit ordinal it
  holds for all $x,y\in S$ that
  \begin{align}
    \label{eq:def_sim_limit_ord}
    &x\sim_\alpha y \quad\text{ if and only if}
    \\ \notag
    &\text{there is a } \beta< \alpha \text{ such that } x\sim_\beta y. 
  \end{align}
  Thus $\sim_\alpha$ has been defined for all ordinals $\alpha$ (by
  transfinite induction). 

  \smallskip
  We now show that all equivalence relations $\sim_\alpha$ are
  $f$-invariant. This holds for $\sim_0$ by assumption. Assume
  $\sim_\beta$ is $f$-invariant for all $\beta< \alpha$. If $\alpha=
  \beta+1$, i.e., if $\alpha$ is not a limit ordinal, it follows from
  Claim~1 above that $\sim_\alpha$ is $f$-invariant. Finally if
  $\alpha$ is a limit ordinal, it follows from the description
  \eqref{eq:def_sim_limit_ord} that $\sim_\alpha$ is $f$-invariant. By
  transfinite induction it follows that $\sim_\alpha$ is $f$-invariant
  for all ordinals $\alpha$. 

  \smallskip
  Clearly $\sim_\alpha$ is closed if and only if $\sim_{\alpha+1}\,=\,
  \sim_\alpha$. In this case the closure of $\sim$ is $\simhat\,=\,
  \sim_\alpha$. Furthermore it then follows that $\sim_\alpha\,=\,
  \sim_\gamma$ for all $\gamma\geq \alpha$. Thus the proof is finished
  with the following.

  \begin{claim2*}
    There exists an ordinal $\alpha$ such that $\sim_{\alpha+1}\,=\,
    \sim_\alpha$. 
  \end{claim2*}
  If this would not be true then all equivalence relations
  $\sim_\alpha$ would be distinct. This is impossible, since
  the cardinality of the set of all equivalence relations on $S$ 
  is bounded by the cardinality of the power set of $S\times S$. 

\end{proof}

\begin{remark}
  The previous proof does not use the axiom of choice, since the
  family of equivalence relations $\sim_\alpha$ constructed in the
  proof is well-ordered by construction. 
\end{remark}

\subsection{Pseudo-isotopies and Moore's theorem}
\label{sec:pseudo-isot-moor}

Moore's theorem is of central importance in the theory of
matings. It gives a condition when an equivalence relation on
the $2$-sphere $\mathbb{S}^2$ yields a quotient space that is again
(homeomorphic to) a $2$-sphere. 

\begin{definition}
  \label{def:Moore-type}
  An equivalence relation $\sim$ on $\mathbb{S}^2$ is called of
  \emph{Moore-type} if 
  \begin{enumerate}[\upshape(1)]
  \item 
    \label{item:Mooretype1}
    $\sim$ is \emph{not trivial}, i.e., there are at least two
    distinct equivalence classes;
  \item 
    \label{item:Mooretype2}
    $\sim$ is \emph{closed};
  \item 
    \label{item:Mooretype3}
    each equivalence class $[x]$ is \emph{connected};
  \item 
    \label{item:Mooretype4}
    no equivalence class \emph{separates} $\mathbb{S}^2$, i.e.,
    $\mathbb{S}^2\setminus [x]$ is 
    {connected} for each equivalence class $[x]$.  
  \end{enumerate}
\end{definition}


\begin{definition}
  \label{def:pseudo-isotopy}
  A homotopy $H \colon X\times [0,1]\to X$ is called a
  \emph{pseudo-isotopy} if $H\colon X\times[0,1) \to X$ is an
  isotopy (i.e., $H(\cdot, t)$ is a homeomorphism for all $t\in
  [0,1)$).  
  We will always assume that $H(x,0)= x$ for all $x\in X$.

  Given a set $A\subset \mathbb{S}^2$, we call $H$
  a pseudo-isotopy \defn{rel.\ $A$} if $H$ is a homotopy rel.\ $A$,
  i.e., if $H(a,t)=a$ for all $a\in A$, $t\in [0,1]$. We call the map $h:=
  H(\cdot, 1)$ the \defn{end} of the pseudo-isotopy $H$. We
  interchangeably write $H(\cdot, t)= H_t(\cdot)$ to unclutter
  notation.   
\end{definition}

The following is \emph{Moore's Theorem}.  
A weaker version was proved by R.~L.~Moore in \cite{Moore}, see also
\cite{Timorin}.

\begin{theorem}[Moore, 1925]
  \label{thm:Moore}
  Let $\sim$ be an equivalence relation on $\mathbb{S}^2$. Then $\sim$ is of
  Moore-type if and only if $\sim$ can be realized as the end of a
  pseudo-isotopy $H\colon 
  \mathbb{S}^2\times [0,1]\to \mathbb{S}^2$ (i.e., $x\sim y
  \Leftrightarrow H_1(x)= H_1(y)$). 
\end{theorem}
The ``if-direction'', i.e., the easy direction, can be found in
\cite[Lemma~2.4]{unmating}. 
A proof of the other direction can be found in
\cite[Theorem~25.1 and Theorem~13.4]{MR872468}. From
Lemma~\ref{lem:sim_induced} we immediately recover the original form of
Moore's theorem, i.e., that $\mathbb{S}^2/\?\sim$ is homeomorphic to
$\mathbb{S}^2$.

\section{Polynomials}
\label{sec:polynomials}

\subsection{Background from complex dynamics}
\label{sec:backgr-from-compl}

Let $R(z) = p(z)/q(z) : \Chat \to \Chat$ be a rational map, 
where the polynomials $p$ and $q$ are without common factors.
The degree of $R$, i.e.,~the maximum of the degrees of $p$ and $q$, 
will be assumed to be at least $2$.
We consider the dynamical system given by iteration of $R$, i.e., with orbits:
$$
z_0, z_1, \ldots, z_n = R(z_{n-1}), \ldots.
$$

A point $z\in \Chat$ is called a \emph{fixed point} if $R(z) = z$. We then
call $\lambda:= DR(z)=R'(z)$ the multiplier. The point $z$ is called super-attracting, 
attracting, neutral, and repelling respectively if $\lambda=0,
0<\abs{\lambda}<1, \abs{\lambda}=1, \abs{\lambda}>1$. 

A point $z$ is
periodic if it is a fixed point for some iterate $R^k$ (for some
$k\geq 1$), and (super-) attracting, neutral, and repelling
if $z$ is a (super-) attracting, neutral, and repelling fixed point of
$R^k$ respectively. 

The Fatou set $F_R$ is the open set of points in $\Chat$, 
for which the family of iterates $\{R^n\}_n$ form a normal family 
in the sense of Montel on some neighborhood of the point.
The Julia set $J_R$ is the compact complement. 
Equivalently the Julia set is the closure of the set of repelling periodic points.

\smallskip
For a polynomial 
$$
P(z) = a_dz^d + a_{d-1}z^{d-1} + \dots + a_1z + a_0
$$ 
the point $\infty$ is a super-attracting fixed point. 
Consequently the Julia set is a compact subset of $\C$. 
The set
$$
K_f = \{z\in\C \mid f^n(z) \not\to\infty, \textrm{as } n\to\infty \}
$$
is called the filled-in Julia set for $f$. Its topological boundary is the Julia set, 
$J_f = \partial K_f$.
The set $J_f$ and hence $K_f$ is connected precisely
if no (finite) critical point escapes or iterates to $\infty$.

\subsection{B\"{o}ttcher's theorem}
\label{sec:bottchers-theorem}

The dynamics of polynomials is much better understood than the
dynamics of general rational maps. The main reason is the following theorem. 

\begin{theorem}[B\"{o}ttcher]
  \label{thm:boettcher}
  Let $P(z) = z^d + a_{d-1}z^{d-1} +\dots+ a_1z + a_0$ be a monic
  polynomial. Then
  \begin{itemize}
  \item $P$ is conformally conjugate to $z^d$ in a neighborhood of
    $\infty$.
  \item If the Julia set $J$ (equivalently the filled-in Julia set
    $K$) of $P$ is connected the conjugacy extends conformally to the
    Riemann map
    $\varphi \colon \Chat \setminus \Dbar \to \Chat \setminus K$. This
    means the following diagram commutes
    \begin{equation*}
      \xymatrix{
        \Chat \setminus \Dbar \ar[r]^{z^d}\ar[d]_{\varphi}
        &
        \Chat\setminus\Dbar\ar[d]^{\varphi}
        \\
        \Chat\setminus K \ar[r]^P 
        &
        \Chat\setminus K.
      }
    \end{equation*}
  \end{itemize}
\end{theorem}

Note that in the above $\varphi(\infty)=\infty$. Furthermore we can and shall 
assume $\varphi$ chosen such that 
$\lim_{z\to \infty}z/\varphi(z) = 1$. 
Note that $G(z):= \log \abs{\varphi^{-1}(z)}$ is the \emph{Green's function} of
the domain $\Chat\setminus K$ (with pole at $\infty$). 


\DEF
The \emph{external ray} $R(\zeta)$ of angle $\zeta \in \Sen =\{\abs{z}=1\}$ is the arc
$$
R(\zeta) := R_P(\zeta) = \varphi(\{r\zeta \mid r>1\})$$
\ENDDEF

\subsection{The Carath\'{e}odory loop}
The following theorem of Carath\'{e}odory is fundamental to the study of the boundary 
of simply connected proper subsets of the plane.
\THM[Carath\'{e}odory]
A univalent map {\mapfromto \phi \D \Chat} has a continuous extension to 
$\Sen=\partial\D$ if and only if $\partial\phi(\D)\subset \Chat$ is
locally connected. 
\ENDTHM
\COR
\label{cor:boettcher_ext}
Let the Julia set of the polynomial $P$ be connected and locally
connected. Then the {\Bottcher} conjugacy $\varphi$ from
Theorem~\ref{thm:boettcher} extends to a continuous map
$$
\overline{\varphi} \colon \Chat \setminus \D \to
\Chat\sm\overset\circ{K}. 
$$
This extension is a semi-conjugacy from $z^d$ to $P$ where defined. 
\ENDCOR

In this case the map $\sigma = \sigma_f:= \varphi|_\Sen \colon \Sen\to \partial
K=J$ is called the \emph{Carath\'{e}odory loop} 
or \emph{Carath\'{e}odory semi-conjugacy} as it semi-conjugates 
$z^d$ to $P$. 
Recall that this means it satisfies the following commutative diagram
\begin{equation}
  \label{eq:Cara_semi}
  \xymatrix{
    \Sen \ar[r]^{z^d} \ar[d]_\sigma & \Sen \ar[d]^\sigma
    \\
    J \ar[r]^P & J,
  }
\end{equation}
i.e., $\sigma(z^d)= P\circ \sigma(z)$ for all $z\in \Sen$. 
It follows from Lemma~\ref{lem:semi_sim} that for the equivalence relation
$\sim~=~\sim_\sigma$ on $\Sen$ induced by $\sigma$ the quotient map 
$z^d/\?\sim\;\colon \Sen/\?\sim\, \to \Sen/\?\sim$ 
is topologically conjugate to $P\colon J\to J$, because $\Sen$ is compact and 
$J\subset \C$ is Hausdorff.

\subsection{The lamination of a polynomial}
\label{sec:lamin-polyn}

Let $P$ be a monic polynomial, $J$ its Julia set, and $K$ its filled
Julia set. We assume that $J$ (equivalently $K$) is connected and
locally connected.  

Denote as above by $\sigma\colon \Sen \to J$ its Carath\'{e}odory semi-conjugacy. 
We have seen in \eqref{eq:Cara_semi} that the equivalence relation
induced by $\sigma$ allows to understand $P\colon J \to J$ as a factor
of $z^d\colon \Sen \to \Sen$. 
There is a closely related construction
that allows to understand $P \colon K \to K$ as a factor of a self-map
of the closed unit disk $\Dbar$, see e.g.~\cite[Theorem 1, page 433]{Douady} 
for a construction of $K$. The success of this approach of studying the topology of,  
and dynamics on the filled Julia sets of polynomials by making pinched disk models of $K$, 
may serve as a motivation for defining and studying matings. 
In fact one way to define matings is to glue two pinched disk models along the circle 
$\Sen$ by the map $z\mapsto \zbar$. 

It will be useful to keep in mind below that the cardinality 
of any equivalence class $[z] \cap\Sen$ is finite. 
In fact the rational points $\exp(i2\pi(\Q/\Z))$ are the (pre)-periodic points in $\Sen$ 
under the map $z \mapsto z^d$. 
The rationals are thus the arguments of the (pre)-periodic external rays for $f$. 
By classical results of Sullivan, Douady and Hubbard, the (pre)-periodic rays land on 
(pre)-periodic points which are either repelling or parabolic. 
And conversely the Douady be-landing Theorem asserts that when $K_f$ is connected, 
every repelling or parabolic (pre)-periodic point is the landing point of at least one 
and at most finitely many (pre)-periodic rays. 
Moreover a theorem of Kiwi, states that the maximum number of non-(pre)-periodic 
rays co-landing on a single non-pre-critical and non pre-periodic point is $d$, \cite{Kiwi}. 
For a more recent and enlarged discussion see also the paper by 
Blokh et al. \cite{Blokhetal}.

\section{Mating definitions}
There are many definitions of mating in use, 
e.g.~topological mating, formal mating, 
intermediate forms such as slow matings as introduced by Milnor 
and explored by Buff and Cheritat (see also the contribution by
Cheritat in this volume) 
and Shishikuras degenerate matings, geometric or conformal mating and 
Douady and Hubbards original definition of mating used by Zakeri-Yampolsky.
We start with the topological mating which is the simplest to formulate and 
which readily exhibits the difficulties related to matings.

\subsection{Definition of the Topological Mating}
\label{sec:defin-topol-mating}

Let $P_\wt$, $P_\bt$, $\wt$ for 'white' and $\bt$ for 'black', 
be two monic polynomials of the same degree $d\geq 2$
with connected and locally connected filled-in Julia sets $K_\wt, K_\bt$. 
We consider the disjoint union $K_\wt \sqcup K_\bt$ and the map
$P_\wt \sqcup P_\bt \colon K_\wt\sqcup K_\bt \to K_\wt \sqcup K_\bt$ 
given by
\begin{equation*}
  P_\wt \sqcup P_\bt|_{K_\wt}= P_\wt,
  \quad P_\wt \sqcup P_\bt|_{K_\bt}= P_\bt.
\end{equation*}

Let $\sigma_j \colon \Sen\to \partial K_j$ be their
Carath\'{e}odory loops (here $j=\wt,\bt$). 
Let $\sim$ be the equivalence relation on $K_\wt\sqcup K_\bt$ generated by
\begin{equation*}
  \sigma_\wt(z)  \sim \sigma_\bt(\bar{z}).
\end{equation*}
for all $z\in \Sen$. 
Note that $\sigma_\wt(z) \in J_\wt\subset K_\wt$,
$\sigma_\bt(\bar{z}) \in J_\bt\subset K_\bt$. 
Furthermore from \eqref{eq:Cara_semi} it follows that
\begin{equation*}
  P_\wt(\sigma(z))= \sigma_\wt(z^d) \sim \sigma_\bt(\bar{z}^d) =
  P_\bt(\sigma(\bar{z})).  
\end{equation*}
Thus $\sim$ is $P_\wt\sqcup P_\bt$ invariant
(see Lemma~\ref{lem:sim_rel_finv}) and hence this map
descends to the quotient, see Section~\ref{sec:equiv-relat-dynam}. 

\begin{definition}[Topological mating]
  \label{def:top_mate}
  Let $\sim$ be as above. Then
  \begin{equation*}
    K_\wt \mate K_\bt := K_\wt \sqcup K_\bt/\sim.
  \end{equation*}
  The \emph{topological mating} of the polynomials $P_\wt, P_\bt$ is
  the map 
  \begin{align*}
    P_\wt\mate P_\bt:= P_\wt\sqcup P_\bt/\?\sim \;
    \colon K_\wt \mate K_\bt \to K_\wt \mate K_\bt.
  \end{align*}
\end{definition}


\subsection{Obstructions and equivalences}
\label{sec:obstr-equiv}

From Definition~\ref{def:top_mate} it looks very surprising that
$P_\wt\mate P_\bt$ 
is ``often'' (topologically conjugate to) a rational map. Rather it
seems that there is no reason to assume that $P_\wt\mate P_\bt$ has
any nice properties at all. We list the different \emph{obstructions}.
\newcounter{temp1}
\setcounter{temp1}{\value{theorem}}
\begin{definition}[Mating obstructions]
  \label{def:obstructions}
  
  \mbox{}
  \begin{itemize}
  \item The equivalence relation $\sim$ may fail to be closed. In this
    case $K_\wt\mate K_\bt$ is not a Hausdorff space (see
    Lemma~\ref{def:closed_eq}). We then say that the mating of $P_\wt,
    P_\bt$ is \emph{Hausdorff-obstructed}.
  \item The space $K_\wt\mate K_\bt$ may fail to be a topological sphere. 
    In this case we call the mating \emph{Moore-obstructed}.
  \item If $P_\wt\mate P_\bt$ is a post-critically finite branched covering, 
    it may fail to be Thurston-equivalent to a rational map, 
    the mating then is \emph{Thurston-obstructed}. 
  \end{itemize}
\end{definition}

If in the last case the Thurston obstruction happens to be a Levy
cycle, we also call the mating \emph{Levy-obstructed}. Note that one
possible obstruction is omitted in the above, since the following
holds. 

\begin{proposition}
  \label{prop:mating_branched_covering}
  Assume the topological mating of $P_\wt, P_\bt$ is not Moore
  obstructed, i.e., $K_\wt\mate K_\bt$ is topologically $\Sto$. Then
  the mating $P_\wt\mate P_\bt\colon \Sto \to \Sto$ is an
  orientation-preserving branched covering. 
\end{proposition}
The proof is postponed to Section~\ref{sec:ray-equivalence}.

The two views on matings is reflected in the following definitions:
\begin{definition}[Matings as rational maps]
  \label{def:matings_mate_rationals}
  Let $P_\wt, P_\bt$ be two degree $d>1$ polynomials for which there exists 
  a homeomorphism {\mapfromto h {K_\wt \mate K_\bt} \Sto}.  
  \begin{itemize}
  \item The polynomials $P_\wt, P_\bt$ are called \emph{combinatorially mateable} 
  if they are post-critically finite and if the branched covering 
  $h\circ P_\wt \mate P_\bt\circ h^{-1}$ is Thurston equivalent to a rational map $R$.
  \item The polynomials $P_\wt, P_\bt$ are called \emph{topologically mateable} 
  if the homeomorphism $h$ can be so chosen that $R=h\circ P_\wt \mate P_\bt\circ h^{-1}$ 
  is a rational map. 
  \item The polynomials $P_\wt, P_\bt$ are called \emph{geometrically mateable} if 
  they are topologically mateable and $h$ can additionally be chosen to be conformal on 
  the interior of $K_\wt \mate K_\bt$. 
  \end{itemize}
Conversely we say that a rational map $R$ is \emph{combinatorially a mating, 
topologically a mating or geometrically a mating} if there exist polynomials 
$P_\wt, P_\bt$ satisfying the corresponding property above with 
$R=h\circ P_\wt \mate P_\bt\circ h^{-1}$. 
\end{definition}

Note that the first two definitions make sense for Thurston maps as
well.
\newcounter{temp2}
\setcounter{temp2}{\value{theorem}}
\setcounter{theorem}{\value{temp1}}
\begin{definition}[\textbf{Cont.}]
  In the notation of \defref{def:matings_mate_rationals}:  
  \begin{itemize}
  \item If $h$ can not be chosen so that $R$ is rational we say that the mating is 
  \emph{topologically obstructed}.
  \item If $h$ can not be chosen so that $R$ is rational and $h$ is conformal on the interior 
  of $K_\wt \mate K_\bt$ we say the mating is \emph{geometrically obstructed}.
  \end{itemize}
\end{definition}
\setcounter{theorem}{\value{temp2}}
Conjecturally any pair of topologically mateable polynomials are geometrically mateable, 
i.e., there are no purely geometrically obstructed pairs. 

It is not known whether there are polynomials $P_\wt, P_\bt$, whose
matings have a Hausdorff obstruction, i.e., which results in an
equivalence relation $\sim$ that is not closed. We can however take
the closure $\simhat$ of $\sim$, and then $P_\wt\mate P_\bt$ descends to the
closure by Proposition~\ref{prop:clos_dynamics}.

\begin{definition}[Closed topological mating]
  \label{def:closed_top_mating}
  Let 
  \begin{equation*}
    K_\wt \widehat{\mate} K_\bt := (K_\wt \sqcup K_\bt)/\simhat.
  \end{equation*}
  The \defn{closed topological mating} is defined as
  \begin{equation*}
    P_\wt \widehat{\mate} P_\bt := (P_\wt \sqcup P_\bt)/\simhat \;\colon
    K_\wt \widehat{\mate} K_\bt \to K_\wt \widehat{\mate} K_\bt.
  \end{equation*}
\end{definition}

However a priori this is not enough to guarantee that the quotient is homeomorphic to $\Sto$.

\subsection{Semi-conjugacies associated to the topological mating}
\label{sec:semi-conj-assoc}

There are several semi-conjugacies naturally associated with matings. 
Assume the rational map $R\colon \Chat \to \Chat$ is topologically 
the mating of the polynomials $P_\wt, P_\bt$. 

We first note that both Julia sets $J_\wt, J_\bt \subset K_\wt \sqcup
K_\bt$ are mapped by the quotient map $K_\wt\sqcup K_\bt \to (K_\wt
\sqcup K_\bt)/\!\sim\; = K_\wt \mate K_\bt$ to the same set. 
In fact the following Lemma is an easy exercise left to the reader:
\begin{lemma}
  \label{lem:mating_semi_conj}
  Let the rational map $R\colon \Chat \to \Chat$ be topologically 
  the mating of the polynomials $P_\wt, P_\bt$. 
  \begin{itemize}
  \item Then the maps {\mapfromto {P_j} {K_j}{K_j}}, $j=\wt,\bt$ 
    are semi-conjugate to $R$ on the image of the quotients $K_j/\!\sim$. More precisely
    there are maps $\varphi_j\colon K_j \hookrightarrow \Chat$
    for $j=\wt, \bt$ (in general neither injective nor surjective) such that 
    \hbox{$\sim_{\varphi_j} =\; \sim|_{K_j}$} and the following
    diagram commutes.
    \begin{equation*}
      \xymatrix{
        K_j \ar[r]^{P_j} \ar@{^{(}->}[d]_{\varphi_j}
        & K_j \ar@{^{(}->}[d]^{\varphi_j}
        \\
        \Chat \ar[r]^R & \Chat
      }
    \end{equation*}
  \item The dynamics of $R$ on its Julia set $J_R$, i.e., $R\colon J_R
    \to J_R$, is a factor of both 
  {\mapfromto {P_j} {J_j} {J_j}}, $j=\wt,\bt$.
  \end{itemize}
\end{lemma}

Topological mating is quite flexible, when the polynomials admits 
dynamics preserving topological deformations, e.g. 
when one of the polynomials has a critical point with an infinite orbit 
contained in hyperbolic components. 
In contrast the conformal mating gives very strong ties between the pair of polynomials 
$P_\wt\mate P_\bt$ and the rational map $R$ realizing the geometrical mating.
The geometrical mating thus gives rise to enumerative type questions:
\begin{itemize}
  \item If two polynomials are geometrically mateable. Is the resulting rational map $R$
  then unique up to {\Mobius} conjugacy? If not are there finitely many such $R$? 
  And then how many? If there are infinitely many, are there natural parametrizations?
  \item How many ways can a rational map be obtained as a geometrical mating of two polynomials?
  \item When does the mating depend continuously/measureably on input data?
\end{itemize} 

\subsection{The Formal Mating}
The fine print of the above discussion is that while the topological
mating is quite easily defined, it is often difficult to
visualize. For example when the filled Julia sets $K_\wt, K_\bt$ are
dendrites it is quite counterintuitive that ``often'' $K_\wt \mate
K_\bt$ is a topological sphere.

The formal mating introduced in the following, circumvents some of the
difficulties. We always obtain a branched covering right from the
start. And taking a further quotient yields again the topological mating. 

Denote by $\Cbar := \C \cup \{(\infty,w)|w\in\Sen\}$ the compactification of $\C$ 
obtained by adjoining a circle at infinity. Note that each monic
polynomial $P= z^d + \dots$ extends continuously to $P\colon \Cbar
\to \Cbar$ by $P((\infty, w)) = (\infty, w^d)$. We denote the extended
map again by $P$ for convenience.

\DEF[Formal Mating]
Let $P_\wt \colon \Cbar_\wt \to \Cbar_\wt$, $P_\bt \colon \Cbar_\bt
\to \Cbar_\bt$ be two monic polynomials of the same degree. Here
$\Cbar_i$ are the compactifications as above of the dynamical planes
$\C_i$  
for each polynomial $P_i$, $i=\wt,\bt$.
Define 
$$
\Cbar_\wt\uplus \Cbar_\bt = (\Cbar_\wt\sqcup \Cbar_\bt)/((\infty_\wt,w)\sim(\infty_\bt,\wbar)).
$$
We write $(\infty,w)$ for the equivalence class containing
$(\infty_\wt,w)\sim(\infty_\bt,\wbar)$ 
and equip $\Cbar_\wt\uplus \Cbar_\bt$ with the quotient topology. We
call this set the \emph{formal mating sphere}.
The
set $\Seninfty:= \{(\infty, w) \mid w\in \Sen\}\subset \Cbar_\wt
\uplus \Cbar_\bt$ is called the \emph{equator} of the formal mating
sphere. 

\smallskip
The \emph{formal mating} 
{\mapfromto {P_\wt \uplus P_\bt} {\Cbar_\wt \uplus \Cbar_\bt} {\Cbar_\wt \uplus \Cbar_\bt}} 
is the map 
$$
P_\wt \uplus P_\bt(z) = 
  \begin{cases}
  P_\wt(z), \qquad &\textrm{if } z\in \C_\wt,\\
  P_\bt(z), \qquad &\textrm{if } z\in \C_\bt,\\
  (\infty,z^d), \qquad &\textrm{for } (\infty,z).
  \end{cases}
$$
\ENDDEF

Evidently $\Cbar_\wt\uplus \Cbar_\bt$ is homeomorphic to $\Sto$ and
$P_\wt \uplus P_\bt$  is a branched covering. 

\begin{definition}[Ray-equivalence]
  \label{def:ray-equivalence}
  The external rays $R_\wt(\zeta)\subset \Cbar_\wt,
  R_\bt(\bar{\zeta})\subset \Cbar_\bt$ of the polynomials $P_\wt, P_\bt$
  are naturally contained in the
  formal mating sphere $\Cbar_\wt \uplus \Cbar_\bt$. The \emph{extended
    external ray} of angle $\zeta\in \Sen$ is the closure of
  $R_\wt(\zeta) \cup R_\bt(\bar{\zeta})$ in the formal mating sphere
  $\Cbar_\wt \uplus \Cbar_\bt$, denoted by $R(\zeta)$. Note that
  $R(\zeta)$ contains the point $(\infty, \zeta)$ of the equator. If the
  filled Julia sets $K_\wt, K_\bt$ are connected and locally connected,
  then $R(\zeta)$ contains exactly one point of the Julia set $J_\wt$,
  as well as exactly one point of the Julia set $J_\bt$. The
  \emph{ray-equivalence} $\sim_\ray$ on the formal mating sphere is defined
  to be the smallest equivalence relation, such that all points of the
  formal mating sphere that are contained in the same extended external
  ray are equivalent. 
\end{definition}

\begin{lemma}
  \label{lem:fmate_not_top_rat}
  The formal mating is never topologically conjugate to a rational
  map.  
\end{lemma}

\begin{proof}
  Consider the point $(\infty, 1)$ on the equator of the formal mating
  sphere. It is also contained in the extended external ray
  $R(1)$. Note that this is a fixed point of the formal mating $P_\wt
  \uplus P_\bt$. As a consequence of B\"{o}ttcher's theorem (Theorem
  \ref{thm:boettcher}), all points in the interior of $R(1)$ (i.e., all
  points except the endpoints) are converging under iteration of
  $P_\wt \uplus P_\bt$ to $(\infty, 1)$. However points in a small
  neighborhood of $(\infty, 1)$ on the equator are repelled from
  $(\infty, 1)$ under iteration, since $P_\wt \uplus P_\bt$ on $\Sen$
  is topologically conjugate to $z^d \colon \Sen \to \Sen$. For
  rational maps such a behavior can only occur at a parabolic fixed
  point. However the same behavior occurs for the $n$-th iterate of
  the formal mating at each point $(\infty, \e^{2\pi i k/(d^n-1)})$ for
  all $k=0, \dots, d^n-1$. A rational map cannot have infinitely
  many parabolic periodic points.
\end{proof}

If we ask whether the formal mating is in some sense a rational map,
we thus need a notion that is weaker than topological conjugacy. The
most successful in this context is Thurston equivalence. Clearly, the formal
mating is post-critically finite if and only if both polynomials $P_\wt,
P_\bt$ are post-critically finite. 
We may then ask if $P_\wt \uplus P_\bt$ is Thurston equivalent to a
rational map. 

\smallskip
The most important aspect of the formal mating however is that we can
reconstruct the topological mating from it.  

\begin{proposition}[Topological mating from formal mating]
  \label{prop:top_vs_formal}
  Let $P_\wt, P_\bt$ be two monic polynomials of the same degree
  $d\geq 2$ with connected and locally connected Julia sets.  
  \begin{enumerate}[\upshape(1)]
  \item
    \label{item:fmate_simF}
    The formal mating $P_\wt \uplus P_\bt\colon \Cbar_\wt \uplus
    \Cbar_\bt$ descends to the quotient $\Cbar_\wt \uplus
    \Cbar_\bt/\!\sim_\ray$.
  \item 
    \label{item:fmate_tmate}
    The quotient map is (topologically conjugate
    to) the topological mating $P_\wt \mate P_\bt$. In particular
    $K_\wt \mate K_\bt$ is homeomorphic to $\Cbar_\wt \uplus
    \Cbar_\bt/\!\sim_\ray$.
  \item 
    \label{item:simF_Moore}
    The space $K_\wt \mate K_\bt$ is a topological sphere if and
    only if the ray-equivalence $\sim_\ray$ is of Moore-type (see
    Definition~\ref{def:Moore-type}).   

  \end{enumerate}  
\end{proposition}

\begin{proof}
  \eqref{item:fmate_simF}
  The formal mating $P_\wt \uplus P_\bt$ clearly maps each extended
  external ray to another extended external ray. It follows that
  $\sim_\ray$ is invariant with respect to $P_\wt \uplus P_\bt$, hence
  this map descends to the quotient $\Cbar_\wt \uplus
  \Cbar_\bt/\sim_\ray$ (see Lemma~\ref{lem:sim_rel_finv}).  

  \smallskip
  \eqref{item:fmate_tmate}
  We identify $K_\wt\sqcup K_\bt$ with $K_\wt \sqcup K_\bt \subset
  \Cbar_\wt \uplus \Cbar_\bt$. Then the map $P_\wt \uplus P_\bt$
  clearly agrees with $P_\wt \sqcup P_\bt$ on $K_\wt\sqcup K_\bt$. 

  Note that the extended external ray $R(\zeta)$ contains the points
  $\sigma_\wt(\zeta)\in J_\wt, \sigma_\bt(\bar{\zeta})\in
  J_\bt$. Thus the equivalence relation $\sim_\ray$ restricted to $K_\wt
  \sqcup K_\bt$ (viewed as subsets of the formal mating sphere) is
  equal to the equivalence relation $\sim$ from the topological
  mating.

  Each point $x\in \Cbar_\wt \uplus \,\Cbar_\bt$ not
  contained in the filled Julia sets $K_\wt, K_\bt \subset \Cbar_\wt
  \uplus \Cbar_\bt$ is contained in one extended external ray. Put
  differently, each equivalence class $[x]_\ray$ of $\sim_\ray$ contains a
  representative $z\in K_\wt \sqcup K_\bt \subset \Cbar_\wt \uplus
  \Cbar_\bt$. The map 
  \begin{equation*}
    h\colon (\Cbar_\wt \uplus \Cbar_\bt)/\!\sim_\ray \;\to
    K_\wt \mate K_\bt
  \end{equation*}
  is now defined as follows.
  Each ray-equivalence class $[x]_\ray \in
  (\Cbar_\wt \uplus \Cbar_\bt)/\!\sim_\ray$ is mapped by $h$ to $[z] \in
  K_\wt \mate K_\bt = (K_\wt \sqcup K_\bt)/\!\sim$, where $z\in K_\wt
  \sqcup K_\bt$ is a representative of $[x]_\ray$. This is a well-defined 
  bijection. Furthermore $h$ conjugates $(P_\wt \uplus P_\bt)/\sim_\ray$
  to the topological mating $P_\wt \mate P_\bt$.   
  
  \smallskip
  One subtle point still needs to be verified however. Namely that the
  quotient topologies induced by $\sim$ and $\sim_\ray$ agree. More
  precisely we need to verify that $h$ maps the quotient topology on
  $(\Cbar_\wt \uplus \Cbar_\bt)/\!\sim_\ray$ to the quotient topology on
  $K_\wt \mate K_\bt = (K_\wt \sqcup K_\bt)/\!\sim$. 

  We denote by $\Pi_\ray \colon \Cbar_\wt \uplus \Cbar_\bt \to (\Cbar_\wt
  \uplus \Cbar_\bt)/\!\sim_\ray$, $\Pi \colon K_\wt \sqcup K_\bt \to K_\wt
  \mate K_\bt = (K_\wt \sqcup K_\bt)/\!\sim$ the quotients maps.  

  Let $U\subset (\Cbar_\wt \uplus \Cbar_\bt)/\!\sim_\ray$ be open. This
  is the case if and only if $[U]_\ray:= \Pi_\ray^{-1}(U) \subset \Cbar_\wt
  \uplus \Cbar_\bt$ is open. Then $[U]_\ray \cap (K_\wt \sqcup K_\bt)$ is
  open. Note that $[U]_\ray \cap (K_\wt \sqcup K_\bt)=
  \Pi^{-1}(h(U))$. Thus $h(U)\subset K_\wt \mate K_\bt$ is open. 

  Now let $h(U)\subset K_\wt \mate K_\bt$ be open. Then
  $\Pi^{-1}(h(U))\subset K_\wt \sqcup K_\bt$ is open. From
  \eqref{eq:Cara_semi} it follows that $\Pi_\ray^{-1}(U)\cap
  \Seninfty\subset \Cbar_\wt \uplus \Cbar_\bt$ is an open subset of
  the equator $\Seninfty$ of the formal mating sphere. From this it
  follows that $\Pi_\ray^{-1}(U)$ is open, hence $U\subset (\Cbar_\wt
  \uplus \Cbar_\bt)/\!\sim_\ray$ is open. Thus the map $h$ is a
  homeomorphism. 

  \smallskip
  \eqref{item:simF_Moore}
  Assume first that $\sim_\ray$ is of Moore type.
  From \eqref{item:fmate_tmate} as well as Moore's theorem
  (Theorem~\ref{thm:Moore}, see also Lemma~\ref{lem:sim_induced}) it
  immediately follows that $K_\wt \mate K_\bt$ is a topological
  sphere. 

  Assume now that $K_\wt \mate K_\bt$ is a topological sphere. 
  Since $(\Cbar_\wt \uplus \Cbar_\bt)/\! \sim_\ray$ is homeomorphic to
  the topological sphere $K_\wt \mate K_\bt$ by
  \eqref{item:fmate_tmate}, it follows that $\sim_\ray$ is not trivial,
  i.e., has at least two distinct equivalence classes. 
  
  Each equivalence class $[x]_\ray$ of $\sim_\ray$ is compact. Otherwise
  $(\Cbar_\wt \uplus \Cbar_\bt)/\!\sim_\ray$ is not Hausdorff,
  contradicting the fact that $\Cbar_\wt \uplus \Cbar_\bt$ is
  homeomorphic to the sphere $K_\wt \mate K_\bt$ by
  \eqref{item:fmate_tmate}.

  Each equivalence class $[x]_\ray$ is connected by construction. 

  Assume now that there is an equivalence class $[x]_\ray$ that separates
  the sphere $\Cbar_\wt \uplus \Cbar_\bt$. Then $[x]_\ray \in (\Cbar_\wt
  \uplus \Cbar_\bt)/\! \sim_\ray$ separates, i.e., $(\Cbar_\wt \uplus
  \Cbar_\bt)/\!\sim_\ray \setminus \{[x]_\ray\}$ is not connected. This shows that $(\Cbar_\wt \uplus
  \Cbar_\bt)/\!\sim_\ray$ is not homeomorphic to the sphere $K_\wt \mate
  K_\bt$.

\end{proof}

Assume now that the two monic polynomials $P_\wt, P_\bt$ of the same
degree $d\geq 2$ have connected and locally connected Julia set.   
Let $\sim_\infty$ be the restriction of $\sim_\ray$ to the equator
$\Seninfty$. This equivalence relation may be described as
follows. Let $\sim_\wt, \sim_\bt$ be the equivalence relations on $\Sen$ induced
by the Carath\'{e}odory loop $\sigma_{\wt, \bt} \colon \Sen \to
J_{\wt, \bt}$, i.e., for all $\zeta, \xi\in \Sen$ it holds
\begin{align*}
  \zeta \sim_\wt \xi \quad &:\Leftrightarrow \quad \sigma_\wt(\zeta)=
  \sigma_\wt(\xi), \text{ and}
  \\
  \zeta \sim_\bt \xi \quad &:\Leftrightarrow \quad \sigma_\bt(\bar{\zeta})=
  \sigma_\bt(\bar{\xi}). 
\end{align*}
Note that from Lemma~\ref{lem:sim_induced} it follows that $\sim_\wt,
\sim_\bt$ are closed. 
Identifying $\Sen$ with $\Seninfty$ it holds that $\sim_\infty$ is the
equivalence relation generated by $\sim_\wt$ 
and $\sim_\bt$ (i.e., the join of  $\sim_\wt, \sim_\bt$ in the
lattice of equivalence relations).  This means that $\zeta\sim_\infty
\xi$ if and only if there is a finite sequence $w_0,\dots, w_N\in
\Sen_\infty$ such that 
\begin{equation*}
  \zeta=w_0 \sim_\wt w_1 \sim_\bt \dots \sim_\wt w_{N-1}\sim_\bt w_N= \xi.
\end{equation*}
Note that in the above we may choose $w_0=w_1$ and/or $w_{N-1} =
w_N$. It is sometimes conceptually easier to deal with $\sim_\infty$
instead of $\sim_\ray$. We have the following.

\begin{lemma}
  \label{lem:sim_inf_ray}
  The ray-equivalence $\sim_\ray$ (on the formal mating sphere $\Cbar_\wt \uplus
  \Cbar_\bt$) is closed if and only if
  $\sim_\infty$ is closed (on $\Seninfty$).  
\end{lemma}

\begin{proof}
  Clearly $\sim_\ray$ being closed implies that $\sim_\infty$ is
  closed. 

  \smallskip
  Assume now that $\sim_\infty$ is closed. Let $(x_n)_{n\in \N},
  (y_n)_{n\in \N}$ be convergent sequences in the formal mating sphere
  $\Sto$ such that $x_n \sim_\ray y_n$ for all $n\in \N$. We need to
  show that $x:= \lim x_n \sim_\ray y:= \lim y_n$. We can assume
  without loss of generality that $x\in \Cbar_\wt\subset\Cbar_\wt
  \uplus \Cbar_\bt$.

  If $x_n=y_n$ for infinitely many $n\in\N$ there is nothing to
  prove. Thus we can assume $x_n \neq y_n$ for all $n\in \N$. This
  implies that $[x_n]_\ray= [y_n]_\ray$ consist of more than a point,
  i.e., contains an extended external ray. Thus 
  $$[x_n]_\infty:=
  [x_n]_\ray \cap \Seninfty\neq \emptyset.$$

  If $x$ is in the interior of $K_\wt$ then $x_n$ is in
  the interior of $K_\wt$ for sufficiently large $n$. By the
  definition of $\sim_\ray$ it follows that $x_n=y_n$ for large
  $n$. Thus $x$ is not in the interior of $K_\wt$ by our
  assumptions. Hence $[x]_\infty= [x]_\ray \cap \Seninfty\neq \emptyset$.

  \smallskip
  We now prove that there is a sequence $(w_n)_{n\in
    \N}\subset\Seninfty$ with $w_n \in [x_n]_\infty$ (for all $n\in \N$)
  such that for a subsequence $(w_{k})_{k\in I}$ ($I\subset \N$
  infinite) it holds
  \begin{equation}
    \label{eq:claimwnk}
    w_{k} \to w  \in [x]_\infty \text{ as } k\in I\to \infty. 
  \end{equation}
  This is clear if $x\notin K_\wt$. Assume now that $x\in
  J_\wt$. Consider the extension of the B\"{o}ttcher map
  $\overline{\varphi}_\wt\colon \Chat \setminus \D\to \Chat \setminus
  \overset{\circ}{K}_\wt$. It is continuous, see
  Corollary~\ref{cor:boettcher_ext}. The equivalence relation on
  $\Chat\setminus \D$ induced by $\overline{\varphi}_\wt$ is still
  denoted by $\sim_\wt$. By Lemma~\ref{lem:sim_induced} $\sim_\wt$ is
  closed.  

  Note that $(\overline{\varphi})^{-1}(x) = [x]_\wt\in \Sen = \partial
  (\Chat \setminus \D)$. Let $U\subset (\Chat\setminus \D)$ be an open 
  neighborhood of $[x]_\wt\in (\Chat \setminus \D)$. Let $U^*$ be the
  saturated interior, i.e., the set of all equivalence classes
  contained (of $\sim_\wt$) in $U$. This is open by
  Lemma~\ref{def:closed_eq}. By definition of the quotient topology it
  follows that $V:=\overline{\varphi}(U^*)\subset \Cbar_\wt$ is an open
  set containing $x$. Since $x_n\to x$ it follows that $x_n\in V$ for
  sufficiently large $n\in \N$. It follows that $[x_n]_\wt$ is
  contained in $U$ (for sufficiently large $n\in\N$). Pick an
  arbitrary $w_n \in [x_n]_\wt$. Since $[x]_\wt$
  is compact, we can extract a convergent subsequence of $(w_n)$
  converging to a point $w\in [x]_\wt$. Finally we note that $[x]_\wt
  \subset [x]_\infty$, finishing the proof of \eqref{eq:claimwnk}. 

  \smallskip
  Using the exact same argument we construct a sequence $u_{k}\in
  [x_{k}]_\infty= [y_{k}]_\infty$ for $k\in I$, and a subsequence
  $(u_{m})_{m\in I'}$ converging to $u\in [y]_\infty$ ($I'\subset
  I\subset \N$ is an infinite set). 

  By construction it follows that $w_m\sim_\infty u_m$  for all $m\in
  I'$. Since $\sim_\infty$ is closed by assumption it follows that
  $w\sim_\infty u$. Thus $[x]_\infty = [w]_\infty=
  [u]_\infty=[y]_\infty$. Thus $[x]_\ray = [y]_\ray$  or $x\sim_\ray
  y$ as desired. 
\end{proof}


\subsection{Shishikuras degenerate Matings}
\label{sec:Shishikuras_degenerate_matings}
It may happen that two post-critically finite polynomials are topologically mateable, 
but that the formal mating is not Thurston equivalent to a rational map. 
More specifically some post-critical points of $P_\wt$ may be ray-equivalent to 
some post-critical points of $P_\bt$. In this case the formal mating will have 
so-called Levi-cycles. Shishikura constructs in the paper \cite{Shishikura} 
a mating intermediate between the topological and the formal mating. 
For this mating the appropriate post-critical points of the formal
mating have been merged, so that the resulting mapping no longer has
any Levi-cycles. 

\subsection{More on the ray-equivalence}
\label{sec:ray-equivalence}

Each ray-equivalence class $[x]_\ray$ may be viewed as a graph. Namely
the set of edges is given by the set of all extended external rays
$R(\zeta)$ contained in $[x]_\ray$. The points in the Julia sets $J_\wt,
J_\bt$ contained in $[x]_\ray$ are the vertices, i.e., the set of
vertices is given by $[x]_\ray \cap (J_\wt \sqcup J_\bt)$. Such a vertex
$v\in [x]_\ray \cap (J_\wt \sqcup J_\bt)$ is incident to the extended
external ray $R(\zeta)\subset [x]_\ray$ (viewed as an edge) if and only
if $v\in R(\zeta)$. Clearly each such graph is connected. 
We call $[x]_\ray$ a \emph{tree} if the just defined
corresponding graph is a tree, otherwise $[x]_\ray$ is called
\emph{cyclic}. 
The diameter of $[x]_\ray$ is defined to be the diameter of the set viewed as a graph. 
If $[x]_\ray$ has infinite diameter, we say $[x]_\ray$ is an
\emph{infinite ray-connection}. Recall from the discussion in
Section~\ref{sec:lamin-polyn} that the number of edges incident to a
vertex is finite. It follows that each $[x]_\ray$ viewed as a graph is
at most countable, i.e., has at most countably many edges. Put
differently $[x]_\infty = [x]_\ray \cap \Seninfty$ is at 
worst countable.

\smallskip
Evidently cyclic ray-equivalence classes do exists: 
Take as $P_\wt$ any polynomial $P$ for which the Carath\'{e}odory loop
$\si\colon \Sen \to J$ is not injective and let 
$P_\bt:= \overline{P(\bar{z})}$. 
On the other hand it is not known whether infinite ray-connections exist.

\RREFPROP{A. Epstein, unpublished}{epstein}
\label{prop:ray-connections}
Let $P_\wt, P_\bt$ be two monic polynomials of the same degree $d\geq
2$ with connected and locally connected Julia sets. Then the following holds.

\begin{itemize}
\item If all equivalence classes of $\sim_\ray$ are ray-trees
and of uniformly bounded diameter, then the topological space
\hbox{$K_\wt \mate K_\bt$} is homeomorphic to $\Sto$.  
\item If $\sim_\ray$ has a cyclic ray-equivalence class or an infinite ray-connection, 
then \hbox{$K_\wt \mate_\ray K_\bt$} is not homeomorphic to $\Sto$.
\end{itemize}
\ENDPROP

\PROOF
To show the first statement we first note that $K_\wt\mate K_\bt$ is 
homeomorphic to $(\Cbar_\wt \uplus \Cbar_\bt)/\!\sim_\ray$ by
Lemma~\ref{prop:top_vs_formal}~\eqref{item:fmate_tmate}. Thus we will
prove the statement by showing that the equivalence relation
$\sim_\ray$ (on the topological sphere $\Cbar_\wt \uplus \Cbar_\bt$)
is of Moore-type. 

\smallskip
Assume now that the equivalence classes $[x]_\ray$ are trees of
uniformly bounded diameter. 
Any non-trivial equivalence class $[x]_\ray$ of bounded diameter is
connected as well as compact, since it is a finite union of compact
sets (i.e., of the extended external rays). Furthermore each
$[x]_\ray$ is by assumption a tree, thus $(\Cbar_\wt \uplus
\Cbar_\bt)\setminus [x]_\ray$ is connected.

We now show that $[x]_\ray$ is closed. By Lemma~\ref{lem:sim_inf_ray}
it is enough to show that the restriction of $\sim_\ray$ to the
equator $\Seninfty$, i.e., $\sim_\infty$, is closed.  Recall that we
assumed that the diameter of the equivalence classes $[x]_\ray$ is
uniformly bounded. This is equivalent to the property that the size of
$[x]_\infty$ is uniformly bounded.

Let $(x_n), (y_n)\in \Seninfty$ be convergent sequences with $x_n
\sim_\infty y_n$ for all $n\in \N$. This means for each $n\in \N$
there are points $w^0_n, \dots, w^N_n\in \Seninfty$  with  
$$x_n=w^0_n\sim_\wt w^1_n \sim_\bt \dots \sim_\bt w^N_n = y_n.$$ 
Taking
a subsequence we can assume that all sequences $(w^j_n)_{n\in\N}$
converge. Since $\sim_\wt, \sim_\bt$ are closed it follows that $x=
\lim x_n \sim_\infty y=\lim y_n$ as desired. 

Since each equivalence class $[x]_\infty$ is finite, it follows that
$\sim_\ray$ is not trivial. Thus $\sim_\ray$ is of Moore-type. By
Moore's Theorem, i.e., Theorem~\ref{thm:Moore}, it follows that
$(\Chat\uplus \Chat)/\sim_\ray$ is homeomorphic to $\Sto$.

\smallskip
We now prove the second statement. First assume that
there exists a cyclic equivalence class $[x]_\ray$. Removing the
point $[x]_\ray$ from the space $(\Chat_\wt
\uplus \Chat_\bt)/\sim_\ray$ yields a disconnected space. Since the
sphere $\Sto$ with a single point removed is 
connected, it follows that $(\Chat_\wt \uplus \Chat_\bt)/\sim_\ray$ is not
homeomorphic to $\Sto$. 

\smallskip
Finally assume that there exists an equivalence class $[x]_\ray$ of
infinite diameter.  Then we can find in $[x]_\ray$ an infinite
ray-path, i.e., a continuous injective map $\gamma\colon [0,\infty)\to
[x]_\ray$ that covers infinitely many (distinct) extended external
rays. Furthermore we choose the parametrization of $\gamma$ as follows
for convenience. For each $n\in \N_0$ the set $\gamma([n,n+1])$ is an extended
external ray, $\gamma(n)\in K_\wt$ for all even $n\in \N_0$,
$\gamma(n)\in K_\bt$ for all odd $n\in \N_0$. Finally we require that
$\gamma(n+1/2)\in \Seninfty$ for all $n\in \N_0$.  

Assume that the restriction of $[x]_\ray$ to the equator $\Seninfty$,
i.e., $[x]_\infty= [x]_\ray \cap \Seninfty$, is not closed. 
Then $\sim_\ray$ is not closed (see
Lemma~\ref{lem:sim_inf_ray}), which implies that $(\Cbar_\wt \uplus
\Cbar_\bt)/\sim_\ray$ is not Hausdorff by Lemma~\ref{def:closed_eq},
hence not homeomorphic to $\Sto$. 
 
Thus we assume now that $[x]_\infty$ is closed. This means that 
all the accumulation points of the sequence 
$\{\ga(n+\frac 1 2)\}_{n\in \N_0} \subset \Seninfty$ 
are contained in $[x]_\infty$. 
In this case note first that by possibly replacing an initial segment of $\ga$ by some 
other ray-path and reparametrizing, we can suppose 
$\ga(\frac 1 2)$ is an accumulation point of the sequence
$\{\ga(n+\frac 1 2)\}_{n\in \N}$. 
However not all the points of this sequence can be accumulation points of the sequence. 
Because then the set of accumulation points of the sequence would contain a Cantor set, 
which by closedness of $[x]_\infty$ would also be a subset of this set, 
contradicting that this set is countable. 
Thus there is some value $m>0$ for which $\ga(m+\frac 1 2)$ is an isolated point 
of $[x]_\infty$.
 
Let $I_0$ and $I_1$ denote the open intervals of $\Seninfty\setminus[x]_\infty$ 
neighboring $\ga(m+\frac 1 2)$. We now show that these two intervals
belong to two disjoint  
connected components of the open set 
$\Cbar_\wt\uplus \Cbar_\bt\Sm [x]_\ray$. 
If not there exists a curve 
{\mapfromto \kappa {[0,1]} {\Cbar_\wt\uplus \Cbar_\bt\Sm [\ga(\frac 1 2)]_\ray}} with 
$\kappa(0)\in I_0$ and $\kappa(1)\in I_1$. 
And such a curve together with the arc of $\Seninfty$ 
between $\kappa(0)$ and $\kappa(1)$ would separate $\ga([0,m+\frac 1 2[)$ from 
$\ga(]m+\frac 1 2,\infty[)$, which contradicts that the later accumulates $\ga(\frac 1 2)$. 
Hence the class $[\ga(\frac 1 2)]_\ray$ is separating and thus 
is an obstruction to $K_\wt \mate_\ray K_\bt$ being homeomorphic to $\Sto$.
\ENDPROOF

\REM
Note that the previous theorem does not cover all cases. Namely the
case when each ray-equivalence class has bounded diameter, but the bound is not
uniform, is not included. This means there is a sequence
$\{[x_n]_\ray\}_{n\in \N}$ of ray-equivalence classes whose diameters
tend to $\infty$. It seems likely 
that in this case one may obtain from a Cantor diagonal type argument
a ray-equivalence class with infinite diameter. We do not have a
proof for this however. 
\ENDREM

We are now ready to prove that in the absence of a Moore obstruction
the topological mating results in a branched covering of the sphere.

\begin{proof}[Proof of
  Proposition~\ref{prop:mating_branched_covering}]

  Let $f =P_\wt \mate P_\bt$ be the topological mating of the monic
  polynomials $P_\wt, P_\bt$ of the same degree $d$, and $K_\wt, K_\bt$
  their filled Julia sets. We assume that $K_\wt \mate K_\bt$ is
  topologically a sphere, which is furthermore identified with
  $(\Cbar_\wt \uplus \Cbar_\bt)/\!\sim_\ray$ by
  Proposition~\ref{prop:top_vs_formal}. The formal mating is denoted
  by $F= (P_\wt \uplus P_\bt)/\!\simr$. 

  \smallskip
  Consider first a point $[y]\in K_\wt \mate K_\bt$ such that the
  corresponding ray-equivalence class $[y]_\ray\subset \Cbar_\wt
  \uplus \Cbar_\bt$ does not contain any critical value of $P_\wt,
  P_\bt$. 

  Let $[x]_\ray$ be a preimage (i.e., a component of the preimage) of $[y]_\ray$ by
  $F$. Note that $[x]_\ray$ does not contain any critical
  point of $P_\wt$ or $P_\bt$. This means that distinct extended
  external rays $R(\zeta), R(\zeta')\subset [x]_\ray$ incident to the
  same vertex $v\in J_\wt \sqcup J_\bt$ are mapped by the formal
  mating $P_\wt\uplus P_\bt$ to distinct external rays. Equivalently
  the angles $\zeta, \zeta'$ are mapped to distinct angles by $z^d$.  

  Let $R(\zeta), R(\zeta')$ be two distinct extended external rays in
  $[x]_\ray$ not intersecting in a vertex $v\in J_\wt \sqcup J_\bt$. 
  If they are mapped to the same extended external ray by $F$
  it follows that $[y]_\ray= F([x]_\ray)$ is cyclic. This cannot happen by
  Proposition~\ref{prop:ray-connections}. Equivalently the angles
  $\zeta, \zeta'$ are mapped to distinct points by $z^d$. This means
  that $[x]_\ray$ is mapped homeomorphically to $[y]_\ray$ by
  $F$. 

  \smallskip
  Consider the restriction to the equator, i.e., $[x]_\infty=
  [x]_\ray \cap\Seninfty$ and $[y]_\infty = [y]_\ray \cap \Seninfty$.  
  It follows from Proposition~\ref{prop:ray-connections}
  that $[y]_\infty$ contains only finitely many points, i.e.,
  $\#[y]_\infty=n <\infty$. Furthermore from
  the above it follows that $F|\Seninfty$ maps $[x]_\infty$ to
  $[y]_\infty$ homeo\-morphically. Recall that $F|\Seninfty$ is
  (conjugate to) the map $z^d \colon \Sen \to \Sen$.
  
  \smallskip
  For each angle $\xi\in \Seninfty$ let $V(\xi)\subset \Seninfty$ 
  be an open arc containing $\xi$ (i.e., an open connected
  neighborhood of $\xi$). Let $V=V[y]= \bigcup_{\xi\in [y]_\infty}
  V(\xi)$. Here the sets $V(\xi)$ in this union are chosen to be
  pairwise disjoint.
  Consider now the set
  $V' = (F|\Seninfty)^{-1} (V)$ (i.e., the preimage of $V$ under the
  map $z^d \colon \Sen \to \Sen$). This 
  set consists of $nd$ open arcs (in $\Seninfty$). Given any point
  $\zeta\in (F|\Seninfty)^{-1}([y]_\infty)$,  
  there is a component $V'(\zeta)$ of $V'$ which is a neighborhood
  of $\zeta$. Conversely each component of $V'$ is a neighborhood
  of such a point $\zeta$. Let $V'[x]:= \bigcup_{\zeta\in [x]_\infty}
  V'(\zeta)$. From the above it follows that $F|\Seninfty\colon V'[x]
  \to V[y]$ is a homeomorphism.

  \smallskip
  We can find an open neighborhood $U=U[y]\subset \Cbar_\wt
  \uplus \Cbar_\bt$ of $[y]_\ray$ such that $U\cap \Seninfty= V$. 
  
  \smallskip
  Consider now the saturated interior $U^*=U^*[y]$ of $U$, i.e., the
  set of all ray-equivalence classes contained in $U[y]$. By
  Lemma~\ref{def:closed_eq}~\eqref{item:closed_saturint} it follows
  that $U^*$ is an open, saturated neighborhood of $[y]_\ray$. It
  follows that each preimage $U^*[x]$ of $U^*$ by $F$ is an open,
  saturated neighborhood of some preimage $[x]_\ray$ of
  $[y]_\ray$ by $F$. Note that $F|\Seninfty$ maps 
  $U^*[x]\cap\Seninfty$, i.e. $V'[x]$,  
  homeomorphically to $U^*\cap\Seninfty$, i.e. to $V$. 
  It follows that $U^*[x]$ is mapped homeomorphically to $U^*$ by $F$.  
 
  Let
  $[U^*]\subset (\Cbar_\wt \uplus \Cbar_\bt)/\!\sim_\ray\,= K_\wt \mate
  K_\bt$ be the image of $U^*$ under the quotient map. This is an open
  neighborhood of $[y]_\ray\in (\Cbar_\wt \uplus
  \Cbar_\bt)/\sim_\ray$. Each component $[U^*[x]]$ of $f^{-1}[U^*]$ is
  the image of a saturated set $U^*[x]$ under the quotient map. Note that $f$
  maps $[U[x]]$ homeomorphically to $[U^*]$. 

  \smallskip
  Let $\mathcal{V}$ be the set of ray-equivalence classes that contain
  critical values, $[\mathcal{V}]\subset (\Cbar_\wt \uplus
  \Cbar_\bt)/\!\simr\,= K_\wt \mate K_\bt$
  be image of this set under the quotient map. We have shown that $f
  \colon (K_\wt \mate K_\bt)\setminus f^{-1}[\mathcal{V}] \to (K_\wt \mate
  K_\bt)\setminus [\mathcal{V}]$ is a covering map. Thus $f$ is a
  branched covering. 

  \smallskip
  That $f$ is orientation preserving follows from the fact the formal
  mating $F$ is orientation-preserving and $f$ may be viewed as a
  pseudo-isotopic deformation of $F$. 
\end{proof}

\subsection{Conformal mating revisited}
There are also stronger notions of Conformal/Geometric Mating in use. 
We start with the original definition of Douady and Hubbard, 
which is easily seen to be equivalent to geometric mateability as defined 
in \defref{def:matings_mate_rationals}. 
This definition was used by Zakeri and Yampolsky 
(see \thmref{thm:zakeriyampolsky} below):

\DEF[Conformal Mating Ia]\label{def:origmatdef}
A rational map {\mapfromto R \Chat \Chat} of degree $d>1$ 
is the conformal mating of two degree $d$ polynomials $P_\wt, P_\bt$ 
with connected and locally connected filled-in Julia sets $K_\wt, K_\bt$, 
if and only if there exists two semi-conjugacies
$$
\phi_i : K_i \to \Chat, \quad\textrm{with}\quad \phi_i\circ P_i = R\circ \phi_i,
$$
conformal in the interior of the filled Julia sets, with 
$\phi_\wt(K_\wt)\cup\phi_\bt(K_\bt) = \Chat$
and with $\phi_i(z) = \phi_j(w)$ for $i,j\in \{\wt,\bt\}$ if and only
if $z\sim w$. Here $\sim$ is the equivalence relation on $K_\wt\sqcup
K_\bt$ which defines the topological mating, i.e., the one defined in
Section~\ref{sec:defin-topol-mating}. 
\ENDDEF

The reader should compare the previous definition with
Lemma~\ref{lem:mating_semi_conj}. 

\RREFDEF{Conformal Mating II}{def:strongconfmat}
A rational map {\mapfromto R \Chat \Chat} of degree $d>1$ 
is the conformal mating of two degree $d$ polynomials $P_\wt, P_\bt$ 
with connected and locally connected filled-in Julia sets $K_\wt, K_\bt$,  
if and only if there exists disjoint conformal embeddings $K_\wt^0, K_\bt^0$ in $\Chat$
of $K_\wt, K_\bt$ and a pseudo-isotopy $H$, 
such that each $H_t$, $t<1$ is conformal on the interior of $K_\wt^0\cup K_\bt^0$ 
and such that $\phi_i=H_1|_{K_i^0}$, $P_i$, $i=\wt,\bt$ and $R$ 
realizes the conformal mating Ia of \defref{def:origmatdef}.
\ENDDEF

At least in case there exists polynomials with locally connected, positive area Julia sets carrying 
an invariant line field, one should consider also the following strengthened version of 
conformal mating.

\RREFDEF{Conformal Mating III}{def:rstrongconfmat}
A rational map {\mapfromto R \Chat \Chat} of degree $d>1$ 
is the conformal mating of two degree $d$ polynomials $P_\wt, P_\bt$ 
with connected and locally connected filled-in Julia sets $K_\wt, K_\bt$,  
if and only if there exists disjoint conformal embeddings $K_\wt^0, K_\bt^0$ in $\Chat$
of $K_\wt, K_\bt$ and a pseudo-isotopy $H$, 
such that each $H_t$, $t<1$ is quasi-conformal with complex dilatation $0$ 
a.e. on $K_\wt^0\cup K_\bt^0$ 
and such that $\phi_i=H_1|_{K_i^0}$, $P_i$, $i=\wt,\bt$ and $R$ 
realizes the Conformal Mating Ia of \defref{def:origmatdef}.
\ENDDEF

\subsection{Mating Questions}
We may summarize the basic mating questions as follows:
\begin{itemize}
  \item
  When is the equivalence relation $\sim_\ray$ closed, i.e., when is there no Hausdorff obstruction?
  \item
  When is there no Moore obstruction?
  \item
  If $K_\wt \mate K_\bt$ is homeomorphic to $\Sto\sim \Cbar$, 
  when is there then a homeomorphism which conjugates $P_\wt \mate P_\bt$ to a rational map?
  \item
  Are the diameters of the equivalence classes of $\sim_\ray$ always finite? 
  Or equivalently are the equivalence classes of $\sim$ always finite?
  \item
  If bounded can they be of arbitrary size?
 (This is the question of existence of long ray-connections)
\end{itemize}  

\subsection{Existence of Matings}
To show that the theory of matings is not vacuous let us mention a few of the 
mating results obtained so far. The first result obtained by M.~Rees,
Tan Lei, and M.~Shishikura (see \cite{MR1149864}, \cite{TanMating},
\cite{Shishikura}) concerns the mating of quadratic post-critically
finite polynomials. 

\THM[Tan Lei, Rees, Shishikura]
Let $P_\wt(z) = z^2 +c_\wt$ and $P_\bt(z) = z^2 +c_\bt$ be two post-critically finite 
quadratic polynomials. 
Then $P_\wt$ and $P_\bt$ are conformally mateable (in the strong sense II) if and only if 
$c_\wt$ and $c_\bt$ do not belong to conjugate limbs of the Mandelbrot set. 
Moreover if mateable the resulting rational map is unique up to {\Mobius}-conjugacy.
\ENDTHM

One implication in the above theorem is relatively easy to see.
Namely if $c_\wt, c_\bt$ belong to conjugate limbs of the Mandelbrot
set, the corresponding polynomials $P_\wt, P_\bt$ are not matable. Let
$c_\wt$ be in the $p/q$ and $c_\bt$ be in the $-p/q$ limb. 
Then the unique $p/q$-cycle of rays for $P_\wt$ is identified with the unique $-p/q$-cycle 
of rays for $P_\bt$. Each ray in the cycles connects a fixed point $\al_\wt$ for $P_\wt$ 
to a fixed point $\al_\bt$ for $P_\bt$. 
The corresponding ray-equivalence class is cyclic, so that the polynomials are 
not even topologically mateable by
Proposition~\ref{prop:ray-connections}. 

Matings of polynomials that are not post-critically finite are much
more difficult to understand. An important result in this setting was proved by
M.~Yampolsky and S.~Zakeri (see \cite{YamZak}). 

\THM[Yampolsky and Zakeri]
\label{thm:zakeriyampolsky}
Suppose $P_\wt,P_\bt$ are quadratic polynomials which are not anti-holomorphically conjugate 
and each with a bounded type Siegel fixed point. 
Then $P_\wt$ and $P_\bt$ are geometrically mateable in the sense of 
\defref{def:origmatdef}.
\ENDTHM 
Here bounded type Siegel fixed point means that the arguments 
$\theta_i\in\R/\Z$ of the corresponding multipliers $\lambda_i=\exp(i2\pi\theta_i)$ 
have continuous fraction expansions with uniformly bounded partial fractions.

On the other hand, regarding the question whether a rational map
arises as a mating of polynomials, we have the following result
obtained in \cite{exp_quotients}, see also \cite{inv_Peano} and
\cite{unmating}. 

\THM[Meyer]
\label{thm:mey_mate}
Let $R\colon \Chat\to \Chat$ be a post-critically finite rational map
such that its Julia set is the whole sphere $\Chat$. Then every
sufficiently high iterate $R^n$ of $R$ arises as a mating (i.e., is
topologically conjugate to the topological mating of two polynomials).
\ENDTHM

In fact the previous statement remains true for \emph{expanding Thurston
  maps}. 

\section{Slow Mating}

Milnor defined in \cite{Milnor:QuadRatMaps} a notion of mating intermediate 
to the topological and the formal mating and depending on a complex parameter 
$\la\in\C\Sm\Dbar$. 
This notion has been explored recently by Buff and Cheritat 
(see also Cheritats contribution to this volume). 

The advantage of Milnors construction is that it constructs directly rational maps 
$\whR_\la$ of the right degree. Moreover the domain and range comes with holomorphically embedded copies of the filled Julia sets. 
The disadvantage is that the domain and the range of the map are not the same and 
in particular the embedded copies of the filled Julia sets in the domain and range are different.
The domain and range can however be identified via a 
$d$-quasi conformal homeomorphism $\chi_\la$, 
which has complex dilatation zero a.e.~on the embedded copies of the filled Julia sets. 
The composition $\wtR_\la=\chi_\la^{-1}\circ \whR_\la$ 
is a degree $d$, $d$-quasi-regular dynamical system, 
which preserves and is conformal on the embedded copies of the filled Julia sets. 
And which is a $d$-quasi-regular 
degree $d$ covering of the separating annulus $C^\la$ to itself.
Moreover the restriction of $\wtR_\la$ to the filled Julia sets has polynomial-like 
(in particular holomorphic) extensions, 
which are conformally conjugate to appropriate polynomial-like restrictions of the polynomials.

\subsection{Definition of the slow mating}
\label{sec:defin-slow-mating}

Let $P_\wt, P_\bt$ be two monic polynomials of degree $d$ with
connected filled Julia sets $K_\wt, K_\bt$. We do however not assume
that $K_\wt, K_\bt$ are locally connected. As in
Theorem~\ref{thm:boettcher} we denote by $\varphi_i\colon \Chat\setminus
\Dbar\to \Chat\setminus K_i$ the B\"{o}ttcher conjugacy (which
conjugates $z^d$ to $P_i$) for $i=\wt, \bt$.  
Recall that the Green's function $G_i$ of $K_i$ (with pole at $\infty$) 
is the subharmonic function given by 
$$
G_i(z) =
\begin{cases}
0,\qquad\textrm{if}\quad z\in K_i\\
\log\abs{\varphi_i^{-1}(z)},\qquad\textrm{if}\quad z\notin K_i.
\end{cases}
$$
For $t > 0$ write
$$
U_i^t := \{z\in\C \;|\; G_i(z) < 2t \}, \qquad i = \wt,\bt
$$
Then {\mapfromto {P_i} {U_i^{(t/d)}} {U_i^t}} is polynomial-like of
degree $d$, i.e.,  
$U'=U_i^{(t/d)} \subset\subset U = U_i^t$ are isomorphic to $\D$ and $P_i$ is 
holomorphic and proper of degree $d$.

Fix any $\lambda\in\C\Sm\Dbar$, write $t =\log|\lambda| > 0$. Consider
the annulus $A=A(\lambda^2)=\{w\in \C \mid 1<w <
\abs{\lambda^2}\}$. 
Note that {\mapfromto{\varphi_i} {A(\lambda^2)}{U^t_i\setminus K_i}} 
is a biholomorphic map. 
And denote by {\mapfromto {\iota_\La}{A(\lambda^2)}{ A(\lambda^2)}} 
the biholomorphic involution $\iota_\la(w)=\lambda^2/w$, which fixes $\la$.  
We identify the two annuli $U^t_\wt\setminus K_\wt, U^t_\bt\setminus K_\bt$ 
via the {\Bottcher} maps $\varphi_i$ and $\iota_\La$, 
so that ``the inner boundary of the first is identified with the outer of the second''
(and vice versa).     
Formally we let $\simla$ be the equivalence relation on 
$U_\wt^t \sqcup U_\bt^t$ generated by 
\begin{equation}
  \label{eq:def_lambda_sim}
  \varphi_\wt(w)\simla \varphi_\bt(\iota_\la(w)),
\end{equation}
for all $w\in A(\lambda^2)$.
Or equivalently for $z_\wt\in U_\wt^t\Sm K_\wt$ and $z_\bt\in U_\bt^t\Sm K_\bt$:
$$
z_\wt\simla z_\bt \quad\Longleftrightarrow\quad
\varphi^{-1}_\wt(z_\wt)\cdot\varphi^{-1}_\bt(z_\bt) = \la^2.
$$
Define $\Chat^\la = (U_\wt^t \sqcup U_\bt^t)/\?\simla$. 
As usual let {\mapfromto {\Pi^\la} {U_\wt^t \sqcup U_\bt^t}{\Chat^\la}} 
denote the natural projection. 
Furthermore equip  $\Chat^\la$ with the complex structure 
given by the two local parameters 
{\mapfromto {\Pi_i^\la=\Pi^\la} {U_i^t} {\Chat^\la}}, $i=\wt,\bt$. 
Then the change of coordinates $U^t_\wt \setminus
K_\wt \to U^t_\bt\setminus K_\bt$  is given by
$\varphi^{-1}_\bt\circ\iota_\la\circ\varphi_\wt(z)$.
The Riemann surface $\Chat^\la$ is simply connected and compact, 
hence isomorphic to $\Chat$. 

\smallskip
For each $\la\in\C\Sm\Dbar$ the polynomials $P_\wt,P_\bt$ induce a proper 
degree $d$ holomorphic map {\mapfromto {R_\la} {\Chat^\la}{\Chat^{\la^d}}} 
given by 
$$
R_\la(w) =
  \begin{cases}
  \Pi_\wt^{\la^d}\circ P_\wt(z), 
  \qquad &\textrm{if } w = \Pi_\wt^\la(z),\\
  \Pi_\bt^{\la^d}\circ P_\bt(z), 
  \qquad &\textrm{if } w = \Pi_\bt^\la(z).
  \end{cases}
$$
This map is well defined and hence proper holomorphic of degree $d$,
since it follows from (B\"{o}ttcher's) Theorem~\ref{thm:boettcher}
that for all $z_\wt\in U^t_\wt\setminus K_\wt, z_\bt \setminus
U^t_\bt$ it holds
\ALIGN
\Pi_\wt^\la(z_\wt) = w= \Pi_\bt^\la(z_\bt)\quad&\Longleftrightarrow\quad 
\varphi^{-1}_\wt(z_\wt)\cdot\varphi^{-1}_\bt(z_\bt) = \la^2\quad \Longrightarrow\\
\varphi^{-1}_\wt(P_\wt(z_\wt))\cdot\varphi^{-1}_\bt(P_\bt(z_\bt)) = \la^{2d} 
\quad&\Longleftrightarrow\quad
\Pi_\wt^{\la^d}\circ P_\wt(z_\wt) = \Pi_\bt^{\la^d}\circ
P_\bt(z_\bt). 
\end{align*}
If we fix conformal isomorphisms {\mapfromto {\eta^\la} {\Chat^\la}{\Chat}}, 
and express $R_\la$ in these coordinates, then $R_\la$ becomes a family of rational maps. 
These are however a priori only defined up to pre- and post-composition 
by {\Mobius} transformations. We shall return to this discussion later.

Again it should be emphasized that this definition is well
defined even for polynomials whose Julia sets are not locally
connected. Thus this may serve as a starting point to define matings
in this setting.

\subsection{Equivalence to the Topological Mating}
\label{sec:equiv-topol-mating}

The topological mating may be recovered from {\mapfromto {R_\la}
  {\Chat^\la}{\Chat^{\la^d}}}, $\la>1$ in a manner analogous to how the
topological mating was obtained from the formal mating in
Proposition~\ref{prop:top_vs_formal}. 

\smallskip
Let $P_\wt, P_\bt$ be monic polynomials of degree $d$ with connected
and locally connected Julia sets. Recall that $R_i(\zeta)\subset
\C\setminus K_i$ denotes the external ray with angle $\zeta\in \Sen$ for
$i=\wt, \bt$. Let $\la$ be real, i.e., $\la>1$. 
Let $\simla_M$ be the smallest equivalence relation on $\Chat^\la$ for which 
for each $\zeta\in\Sen$ the set 
$$
  \overline{\Pi^\la(R_\wt(\zeta)\cap U_\wt^t)} = 
  \overline{\Pi^\la(R_\bt(\overline{\zeta})\cap U_\bt^t)}
$$
is contained in one equivalence class.
Then $R_\la$ semi-conjugates the equivalence relation $\simla_M$ to $\simlad_M$, 
i.e. $z\simla_M z'$ implies $R_\la(z)\simlad_M R_\la(z')$.

Let $K_\wt \mate_M^\la K_\bt = \Chat^\la/\?\simla_M$ with natural projections 
{\mapfromto {\Pi^\la_M} {\Chat^\la} {K_\wt \mate^\la_M K_\bt}} and 
define for $\la>1$ 
$$
{\mapfromto {P_\wt \mate^\la_M P_\bt} {K_\wt \mate^\la_M K_\bt} 
{K_\wt \mate^{\la^d}_M K_\bt}} 
$$
as the mapping induced by $R_\la$, i.e., the top square below commutes
\begin{equation}
\label{eq:Slow_equiv_Top}
  \begin{CD}
  {\Chat^\la} @>{R_\la}>> {\Chat^{\la^d}}\\
  @V{\Pi^\la_M} VV      @VV{\Pi^{\la^d}_M}V\\
  {K_\wt \mate^\la_M K_\bt} @>{P_\wt \mate^\la_M P_\bt}>> 
  {K_\wt \mate^{\la^d}_M K_\bt}
  \\
  @V{h_\la} VV      @VV{h_{\la^d}}V\\
  {K_\wt \mate K_\bt} @>{P_\wt \mate P_\bt}>> {K_\wt \mate K_\bt}.
  \end{CD}
\end{equation}
Then for $\la>1$ the spaces $K_\wt \mate^\la_M K_\bt$ 
are canonically homeomorphic to $K_\wt \mate K_\bt$ by homeomorphisms 
$h_\la$ such that $h_{\la^d}\circ ({P_\wt \mate^\la_M P_\bt}) \circ h_\la^{-1} 
= P_\wt \mate P_\bt$ is the topological mating (i.e., ~the bottom square commutes).
The proof is similar to that of \lemref{prop:top_vs_formal} and is left to the reader.

\subsection{Choosing conformal isomorphisms}
In this section we show that we may choose the isomorphisms 
{\mapfromto {\eta^\la} {\Chat^\la}{\Chat}} such that 
the images of the filled Julia sets move holomorphically in $\Chat$
with $\la\in\C\Sm\Dbar$.  

\smallskip
To normalize the motion we will choose three distinguished points in
$\Chat^\lambda$. Fix two points $z_\wt\in K_\wt$, $z_\bt\in K_\bt$ 
and let $w_\wt(\lambda):=
\Pi^{\la}(z_\wt)$, $w_\bt(\lambda) :=
\Pi^{\la}(z_\bt)$ be their images in $\Chat^\lambda$. 
The third point is chosen as follows. Recall
from \eqref{eq:def_lambda_sim} that $\varphi_\wt(\lambda)
\overset{\lambda}{\sim} \varphi_\bt(\lambda^2/\lambda) =
\varphi_\bt(\lambda)$. Let $w_1(\lambda):=
\Pi^{\la}(\varphi_\wt(\la))=\Pi^{\la}(\varphi_\bt(\la))$. 

We shall choose $\la_0 = \e^1$ as the base point of the motion. 
For $\La>0$ and $K_i$ locally connected we
can choose the points $w_\wt(\lambda_0), w_\bt(\lambda_0),
w_1(\lambda_0)\in \Chat^{\lambda_0}$ so that they are not identified under the
equivalence relation from Section~\ref{sec:equiv-topol-mating}, i.e.,
are in different equivalence classes of $\simlaz_M$. 
That is $\varphi_\wt(\e)=\varphi_\bt(\e), z_\wt, z_\bt$ are
not ray-equivalent in the sense of Definition~\ref{def:ray-equivalence}. 
This is however not necessary for most of the following discussion.

\smallskip
For $\la\in\C\Sm\Dbar$ let {\mapfromto {\eta_\la} {\Chat^\la} \Chat}
be the conformal isomorphism normalized by
$$
\eta_\la(w_\wt(\lambda)) = 0,\quad \eta_\la(w_\bt(\lambda)) = \infty, \textrm{ and } 
\eta_\la(w_1(\lambda)) = 1, 
$$
and define $K_\wt^\la = \eta_\la(\Pi^\la(K_\wt))$, $K_\bt^\la = \eta_\la(\Pi^\la(K_\bt))$.

\REFTHM{thm:Kholomorphicmotion}
The map {\mapfromto M {(\C\Sm\Dbar)\times(K_\wt^{\la_0}\cup K_\bt^{\la_0})} {\Chat}} 
given by 
$$M(\la,z) = \eta_\la\circ\Pi^\la\circ(\Pi^{\la_0})^{-1}\circ\eta_{\la_0}^{-1}(z)
$$ 
is a holomorphic motion with base point $\la_0$, i.e., 
\ENUM
\item
$M(\lambda_0,\cdot)= \Id$ on $K_\wt^{\la_0}\cup K_\bt^{\la_0}$.
\item
For each fixed $z\in K_\wt^{\la_0}\cup K_\bt^{\la_0}$ the map 
$\C\Sm\Dbar\ni\la\to M(\la,z)$ is holomorphic.
\item
For each fixed $\la\in\C\Sm\Dbar$ the map 
$K_\wt^{\la_0}\cup K_\bt^{\la_0}\ni z\mapsto M(\la,z)$ is injective.
\ENDENUM
\ENDTHM
\COR
The family of rational maps 
{\mapfromto {\whR_\la = \eta_{\la^d}\circ R_\la\circ\eta_\la^{-1}} \Chat \Chat} 
depends holomorphically on $\la\in\C\Sm\Dbar$.
\ENDCOR
\PROOF
The zeros and poles of $\whR$ depend holomorphically on $\la$ and $\whR$ fixes $1$.
\ENDPROOF
\REFCONJ{conj:Milnormating}
If the family of degree $d$ rational maps $\whR_\la$, $\la>1$ 
has a limit $R$ of degree $d$, as $\la \searrow 1$, 
then $R$ is a conformal mating of $P_\wt$ and $P_\bt$ in the strongest sense, 
\defref{def:rstrongconfmat}.
\ENDCONJ
Before we proceed to a proof of the theorem, let us introduce a few facts about polynomials 
with connected filled Julia set.
It is well known see e.g. \cite{BrannerandHubbard} or \cite{PetersenandTan} 
that the almost complex structures on $\Chat$ which are given by 
the Beltrami forms 
$$
\sigma_\La =\mu_\Lambda(z)\frac{\dzbar}{\dz}, 
\qquad  
\mu_\Lambda(z) = \frac{\Lambda-1}{\Lambda+1}\frac{z}{\zbar}, \quad z\in\Cstar
$$
are invariant under $z\mapsto z^k$ for every $\Lambda\in\Hplus$ and 
every $k\in\Z\Sm\{0\}$ and under $z\mapsto\al z$ for every $\al\in\Cstar$.  
(Note that contrary to most conventions $\si_1\equiv 0$.)
Moreover the integrating q-c homeomorphism for $\sigma_\La$, 
that fix $0$, $1$ and $\infty$ is the map $\zeta_\La(z) = z|z|^{(\Lambda-1)}$.
It restricts to the identity on the unit circle and conjugates $z\mapsto z^d$ to itself. 
In fact the lift to the logarithmic coordinate on $\Cstar$ which fixes $0$ is the 
real-linear map fixing $i$ and sending $1$ to $\Lambda$. 
The maps $\zeta_\La$ form a group with neutral element $\zeta_1=\Id$, with 
$\zeta_{\La+it}=\zeta_\La\circ\zeta_{1+it}$ for all $\La\in\H$ and $t\in\R$ and 
with $\zeta_{ss'} = \zeta_s\circ\zeta_{s'}$ for all $s, s' >0$. 
And thus defines a group action on $\Chat$.
Consequently for every polynomial $P$ with connected filled Julia set $K$, 
and every $\La\in\H$ we obtain a $P$-invariant almost complex structure on $\Chat$ 
with Beltrami form $\sigma^P_\Lambda=\mu^P_\Lambda\frac{\dzbar}{\dz}$, 
where $\sigma^P_\La$ and hence $\mu^P_\La$ is equal to $0$ on $K_P$ and equal to 
$\psi_P^*(\sigma_\Lambda)$ on $\C\Sm K$, 
i.e.
$$
\mu^P_\La = 
\frac{\Lambda-1}{\Lambda+1}
\frac{\psi_P(z)\overline{\psi_P'(z)}}{\overline{\psi_P(z)}\psi_P'(z)},
\quad c\in\C\Sm K,
$$
where $\psi_P=\varphi_P^{-1}$. 
For this almost complex structure the map 
{\mapfromto {\zeta_\Lambda^P} \Chat\Chat} given by
$$
\zeta_\Lambda^P(z) =
\begin{cases}
z,\quad &z\in K_P\\
\varphi_P\circ\zeta_\La\circ\psi_P(z),\quad &z\notin K_P
\end{cases}
$$
is continuous and hence an integrating q-c homeomorphism. 
It acts on points in $\C\Sm K$ by multiplying potential $t'$ by $\Re(\Lambda)$ 
and adding $\Im(\Lambda)t'$ to the argument. 
In particular for $\Lambda$ real $\zeta_\Lambda^P$ 
preserves rays and maps points of potential $t'$ 
to points of potential $\Lambda t$. 
By construction the maps $\zeta_\Lambda^P$ form a group under composition. 
And this group is canonically isomorphic to the group formed by the maps 
$\zeta_\La$ under composition, as follows from the formula above (see
also \cite{PetersenandTan} for further details).  

Fix $\lambda_0 = \e^1$ and consider the almost complex structures on 
$\Chat^{\lambda_0}$ given by the Beltrami forms 
$\sigma^{\la_0}_\La = (\Pi^{\la_0}_i)_*(\sigma^{P_i}_\La)$,  
which is supported only on the separating annulus 
$C^{\la_0} = \Pi^{\la_0}_i(U^1_i)$, $i=\wt,\bt$. 
Note that the invariance of $\sigma_\La$ under the involutions $\iota_\la$, 
$\sigma_\Lambda = \iota_\la^*(\sigma_\Lambda)$ 
ensures that $\sigma^{\la_0}_\Lambda$ is well defined. 
Moreover $\sigma^\laz_\La$ depends complex analytically on $\Lambda\in\Hplus$, 
since point wise the coefficient function $\mu^\laz_\La$ is a complex scalar multiple 
of norm $1$ or $0$ of the constant $(\La-1)/(\La+1)$.

For each $\La\in\Hplus$ write $\la = \e^\Lambda$ and let 
{\mapfromto {\phi_\La} {\Chat^{\la_0}} \Chat} 
be the integrating homeomorphism for $\sigma^{\la_0}_\Lambda$ 
which is normalized by 
$$
\phi_\La(w_\wt(\laz)) = 0,\quad \phi_\La(w_\bt(\laz)) = \infty, \textrm{ and } 
\phi_\La(w_1(\laz)) = 1.
$$
Then $\phi_1=\eta_{\lambda_0}$ and 
$\phi_\Lambda$ depends holomorphically on $\Lambda$, 
by the Ahlfors-Bers Theorem for almost complex structures depending 
analytically on a complex parameter. 
\REFTHM{thm:fullholomorphicmotion}
The map {\mapfromto H {\Hplus\times\Chat} {\Chat}} 
given by $H(\Lambda,z) = \phi_\Lambda\circ\phi_1^{-1}(z)$ is a holomorphic motion 
with base point $1$ and its ($z$-variable) restriction to 
$(K_\wt^{\la_0}\cup K_\bt^{\la_0})$ is a $2\pi i$ 
periodic holomorphic motion (in the $\Lambda$-variable). 
\ENDTHM
\PROOF
The only statement not justified already is the $2\pi i$ periodicity. 
However for $\La' = 1+2\pi i$ the integrating map $\zeta_{\La'}$ 
for $\si_{\Lambda'}$ restricts to the identity 
on the circles of center $0$ and radius $\e^k$, $k\in\Z$. 
As a consequence $\zeta_{\La'}^{P_i}$ restricts to the identity on $K_i$, on 
the equipotential $\varphi_i(\e\Sen)$ containing $\varphi_i(\laz)$ and on the 
boundary of $U_i^1$ for $i=\wt, \bt$. 
It follows that the map {\mapfromto {\phi} {\Chat^{\la_0}}{\Chat^{\la_0}}} 
given by 
$$
\phi(z) = \Pi^\laz_i\circ\zeta^{P_i}_{\La'}\circ(\Pi^\laz_i)^{-1}(z)
$$
is an integrating qc-homeomorphism for $\si_{\Lambda'}^{\la_0}$ 
which is the identity on $\Pi^{\la_0}(K_\wt\sqcup K_\bt)$ union the 
core geodesic $\Pi_i^\laz(\varphi_i(\e\Sen))$ of the separating annulus $C^\laz$. 
In particular $\phi$ fixes $w_1(\laz), w_\wt(\laz)$ and  $w_\bt(\laz)$.
Thus $\phi_{\La'} = \phi_1\circ\phi$ and $\phi_{\La'} = \phi_1$ on 
$\Pi^{\la_0}(K_\wt\cup K_\bt)$. The $2\pi i$ periodicity then follows because 
the maps $\zeta_\La$ form a group action on $\Cstar$, so that 
$\phi_{\La'} = \phi_1\circ\phi$ implies 
$\phi_{\La+i2\pi} = \phi_\La\circ\phi$ for all $\La\in\H$.
\ENDPROOF
\REFCOR{cor:lambdaKmotion}
The map {\mapfromto {\widehat{H}} 
{(\C\Sm\Dbar)\times(K_\wt^{\la_0}\cup K_\bt^{\la_0})} \Chat} given by 
$\widehat{H}(\la,z) := H(\log\la,z)$ is a holomorphic motion with base point $\la_0$.
\ENDCOR
\REFTHM{thm:equalityonK}
For every $\Lambda\in\Hplus$ and $\la=\e^\Lambda$ 
$\eta_\la\circ\Pi^\la=\phi_\La\circ\Pi^{\la_0}$ on $K_\wt\sqcup K_\bt$ 
and the map 
{\mapfromto {\eta_\la^{-1}\circ\phi_\La} {\Chat^{\la_0}}{\Chat^{\la}}}
is a quasi conformal homeomorphism with complex dilatation $0$ a.e.~on 
$\Pi^{\la_0}(K_\wt\sqcup K_\bt)$. 
Moreover for $\La>0$ and $K_\wt, K_\bt$ locally connected 
this map preserves the equivalence classes of $\simla_M$.
\ENDTHM
\PROOF
For $i=\wt,\bt$ and $t=2\Re(\La)$ the maps 
$$
{\mapfromto {\xi_i=\phi_\La\circ\Pi_i^{\la_0}\circ(\zeta_\Lambda^{P_i})^{-1}}
{U_i^t} {\Chat}}
$$ 
are quasi conformal with complex dilatation $0$ and hence biholomorphic. 
Moreover they satisfy
$$
\forall z_\wt\in U^t_\wt\Sm K_\wt,\forall z_\bt\in U^t_\bt\Sm K_\bt:\qquad
\xi_\wt(z_\wt)=\xi_\bt(z_\bt) \Longleftrightarrow
\varphi_\wt(z_\wt)\varphi_\bt(z_\bt)= \la^2
$$
and $\xi_i=\phi_\La\circ\Pi_i^{\la_0}$ on $K_i$, 
because $\zeta_\La^{P_i} = \Id$ on the filled Julia sets $K_i$. 
Thus it follows from the normalizations that $\xi_i=\eta_\la\circ\Pi_i^\la$ 
so that
$\eta_\la\circ\Pi^\la=\phi_\La\circ\Pi^{\la_0}$ on $K_\wt\sqcup K_\bt$ 
and hence the map 
{\mapfromto {\eta_\la^{-1}\circ\phi_\Lambda} {\Chat^{\la_0}}{\Chat^{\la}}}
is a quasi conformal homeomorphism, 
which has complex dilatation $0$ a.e.~on this set. 
Because $\zeta_\Lambda^{P_i} = \Id$ on the filled Julia sets $K_i$. 
$$
 \xymatrix{
    U^1_i, K_i \quad\ar[r]^{\zeta_\Lambda^{P_i}, \Id}
    \ar[d]_{\Pi_i^\laz}
    &
    \quad U_i^t, K_i\ar[d]^{\Pi_i^\la}
    \\
    \Chat^\laz, \Pi_i^\laz(K_i)\ar[d]_{\phi_1=\eta_\laz}
    \ar[r]^{\eta_\la^{-1}\circ\phi_\La}\quad\ar[dr]^{\phi_\La} 
    &
    \Chat^\la, \Pi_i^\la(K_i)\ar[d]^{\eta_\la}\\
    \Chat, K_i^\laz
    &
    \Chat, K_i^\la.
    }
$$
Finally preservation of $\simla_M$ follows, because $\zeta_\Lambda$ 
preserves lines through the origin, when $\Lambda>0$.
\ENDPROOF
\PROOF (of \thmref{thm:Kholomorphicmotion}):
Theorem~\ref{thm:Kholomorphicmotion} follows immediately by combining 
\corref{cor:lambdaKmotion} and \thmref{thm:equalityonK} since for 
$z\in K_\wt^{\la_0}\cup K_\bt^{\la_0}$
$$
M(\la,\cdot) = \eta_\la\circ\Pi^\la\circ(\Pi^{\la_0})^{-1}\circ\eta_{\la_0}^{-1} =
\phi_\La\circ\Pi^\laz\circ (\Pi^{\la_0})^{-1}\circ\phi_1^{-1} = 
\phi_\La\circ\phi_1^{-1}.
$$
\ENDPROOF
\REFTHM{thm:BCcomparison}
For every $\Lambda>0$ and $\la=\e^\Lambda$ the map 
{\mapfromto {\chi_\la} {\Chat}{\Chat}} given by 
$\chi_\la = \varphi_{d\Lambda}\circ\varphi_\Lambda^{-1}$ is a 
$d$-quasiconformal homeomorphism, with 
$$
\chi_\la(K_\wt^\la\cup K_\bt^\la) = 
K_\wt^{\la^d}\cup K_\bt^{\la^d}
$$ 
and with complex dilatation $0$ almost everywhere on $K_\wt^\la\cup K_\bt^\la$. 
\ENDTHM
\PROOF
For $\La>0$ the maps $z\mapsto z|z|^{(\Lambda'-1)}$ form a commutative 
group action on $\Cstar$. 
Thus it follows that the real dilatation of the qc-homeomorphism 
$\chi_\la = \phi_{d\Lambda}\circ\phi_\Lambda^{-1}$ equals 
the dilation of $\phi_d$ at corresponding points. 
This is $0$ on $\Pi^{\la_0}(K_\wt\sqcup K_\bt)$ 
and an easy computation shows that it equals $d$ 
on the separating annulus $C^{\laz}$.
\ENDPROOF
\COR
For $\La>0$ and $\la=\e^\la$ the  map 
{\mapfromto {\wtR^\la:=\chi_\la^{-1}\circ \whR_\la}{\Chat}{\Chat}} 
is a quasi-regular degree $d$ branched covering which is conformal a.e.~on 
$K_\wt^\la\cup K_\bt^\la$ and which is a $d$-qc covering map from 
$\eta_\la(C^\la)$ to itself. 
Moreover for $i=\wt,\bt$ the map $\wtR^\la$ coincides on $K_i^\la$ 
with the quadratic-like map 
$$
\mapfromto {\wtP^\la_i}{\eta_\la(\Pi^\la_i(U_i^{\La/d}))}
{\eta_\la(\Pi^\la_i(U_i^{\La})) = K^\la_i\cup\eta_\la(C^\la)}
$$
which is conformally conjugate from $P_i$ by $\eta_\la\circ\Pi^\la_i$.
\ENDCOR
Let $\sim_1$ be the equivalence relation on $\Chat$ conjugate by $\eta_{\laz}^{-1}$ 
to $\simlaz$ on $\Chat^{\la_0}$, i.e. $z_1\sim_1 z_2$ if and only if 
$\eta_{\laz}^{-1}(z_1)\simlaz \eta_{\laz}^{-1}(z_2)$. 
In the following proposition we reverse the time orientation of the Moore-isotopy, 
in order to unclutter notation
\PROP
If the restriction {\mapfromto {H} {]0,1]\times\Chat} {\Chat}}  
of the holomorphic motion in \thmref{thm:fullholomorphicmotion}, 
has a continuous extension {\mapfromto \wtH {[0,1]\times\Chat}\Chat}, 
which is a time reversed Moore isotopy for $\sim_1$. 
Then both $\whR_\la$ and $\wtR_\la$, $\la=\e^\Lambda$ converge 
to a (the same) rational map $R_0$ of degree $d$ as $\Lambda\to 0$.
And $R_0$ is a conformal mating of $P_\wt$ and $P_\bt$ in the strong sense of 
\defref{def:rstrongconfmat}.
\ENDPROP
\PROOF
For each $t\in\;]0,1]$ define {\mapfromto {\wtH_t} {[0,1]\times\Chat} \Chat} by 
$\wtH_t(s,z) = \wtH(st,\phi_1\circ\phi_t^{-1}(z)) = \phi_{st}\circ\phi_t^{-1}$. 
Then each $\wtH_t$ is a Moore isotopy for the equivalence relation $\sim_t$ on $\Chat$ 
conjugate to $\simla$ by $\eta_\la^{-1}$, where $\la=\e^t$. 
Let {\mapfromto {\Theta_\la} {K_\wt \mate^\la K_\bt} \Cbar} be the induced
homeomorphism as given by \lemref{lem:sim_dyn}. 
Thus for $0<t<1/d$ we have the commutative diagrams
\begin{equation}
\label{eq:convergence_of_R}
  \begin{CD}
\Chat @<{\eta_\la}<< {\Chat^\la} @>{R_\la}>> {\Chat^{\la^d}} @>\eta_{\la^d}>> \Chat\\
@V{\wtH_t(0,\cdot)}VV  @V{\Pi^\la} VV      @VV{\Pi^{\la^d}}V @VV\wtH_{dt}(0,\cdot)V\\
\Chat @<\Theta_\la<<  {K_\wt \mate^\la K_\bt} 
@>{P_\wt \mate^\la P_\bt}>> {K_\wt \mate^{\la^d} K_\bt} @>\Theta_{\la^d}>> \Chat\\
@.  @V{h_\la} VV      @VV{h_{\la^d}}V @.\\
{} @.  {K_\wt \mate K_\bt} @>{P_\wt \mate P_\bt}>> {K_\wt \mate K_\bt} @. {}
  \end{CD}
\end{equation}
The homeomorphisms 
{\mapfromto {h_\lambda\circ\Theta_\lambda^{-1}} \Chat {K_\wt \mate K_\bt}} are independent of 
$\la$, because for any $z\in K_\wt\sqcup K_\bt$ :
$\wtH_t(0,\eta_\la\circ\Pi^\la(z)) = \wtH(0,\eta_{\la_0}\circ\Pi^{\la_0}(z))$ 
is independent of $\la$. 
Thus by \propref{prop:mating_branched_covering} the map 
$$
R_0=\Theta_{\la^d}\circ h_{\la^d}^{-1}\circ {(P_\wt \mate P_\bt)}\circ h_\la\circ\Theta_\la^{-1}
$$
is a fixed degree $d$ branched covering. 
The map $\wtH$ is uniformly continuous, because it is continuous with compact domain. 
Thus the projections $\wtH_t(0,z)$ converge uniformly to the identity. 
And the holomorphic degree $d$ branched coverings $\whR_\la$ 
converge uniformly to $R_0$, which shows convergence and that $R_0$ is holomorphic. 
Similarly $\chi_\la^{-1}(z)= \wtH_{d\La}(1/d,z)$ converge uniformly to the identity, 
so that $\wtR_\la=\chi_\la^{-1}\circ\whR_\la$ also converges uniformly to $R_0$.
By construction $\wtH\circ\eta_{\la_0}\circ\Pi^{\la_0}_i$, $P_i$ and 
$R_0$ satisfies the requirements of the strong mating definition \defref{def:rstrongconfmat}.
\ENDPROOF

\subsection{Cheritat movies} 
It is easy to see that $\whR_\la$ converges uniformly to the monomial $z^d$ 
as $\la\to\infty$. 
Cheritat has used this to visualize the path of Milnor intermediate matings $\whR_\la$, 
$\la \in ]1,\infty[$ of quadratic polynomials through films. 
Cheritat starts from $\la$ very large so that 
$K_\wt^\la$ and $K_\bt^\la$ are essentially just two down scaled copies 
of $K_\wt$ and $K_\bt$, the first near $0$, the second near $\infty$. 
From the chosen normalization and the position of the critical values in 
$K_\wt^\la\cup K_\bt^\la$ he computes $\whR_{\sqrt{\la}}$. 
From this $K_\wt^{\sqrt{\la}}\cup K_\bt^{\sqrt{\la}}$ can be computed
by pull back of $K_\wt^\la\cup K_\bt^\la$ under $\whR_{\sqrt{\la}}$. 
Essentially applying this procedure iteratively one obtains a sequence of rational 
maps $\whR_{\la_n}$ and sets $K_\wt^{\la_n}\cup K_\bt^{\la_n}$, 
where $\la_n\searrow 1$ and $\la_n^2 = \la_{n-1}$. 
For more details see the paper by Cheritat in this volume.

\section{Appendix: Branched Coverings}
\DEF
A branched covering {\mapfromto F \Sto \Sto} is a map such that:  
For all $x\in\Sto$ there exists local coordinates {\mapfromto \eta {\omega(x)} \C}, 
{\mapfromto \zeta {\omega(F(x))} \C} and $d\geq 1$ such that 
$$
\zeta\circ F \circ \eta^{-1}(z) = z^d.
$$
When $d>1$ above the point $x$ is called a critical point. The set of critical points for $F$ 
is denoted $\Omega_F$.

The branched covering $F$ is called post-critically finite (PCF), 
if the post-critical set
$$
P_F = \{ F^n(x)| x \in \Omega_F, n>0\}
$$
is finite.
\ENDDEF

\subsection{Thurston Equivalence}

\DEF[Thurston Equivalence]
Two post-critically finite branched coverings {\mapfromto {F_1, F_2} \Sto \Sto} 
are said to be \emph{Thurston equivalent} if and only if there exists a pair of 
homeomorphisms {\mapfromto {\Phi_1, \Phi_2} \Sto \Sto} 
isotopic relative to the post critical set of $F_1$ such that 
$$
\begin{CD}
\Sto @>F_1>> \Sto\\
@V \Phi_1 VV      @VV \Phi_2 V\\
\Sto @>F_2>> \Sto.
\end{CD}
$$
\ENDDEF

\subsection{Multicurves}
Let $P\subset\Sto$ be a finite set.
\begin{itemize}
  \item
  A simple closed curve {\mapfromto \ga \Sen {\Sto\sm P}} is called \emph{peripheral} 
  if one of the complementary components $\Sto\sm\ga$ contains at most one point of $P$.
  \item
  A \emph{multi curve} $\Gamma$ in $\Sto\sm P$ is a set or collection of mutually non homotopic, 
  non-peripheral simple closed curves in $\Sto\sm P$.
  \item
  Note that a multi curve has at most $\#P-3$ elements.
\end{itemize}  
\subsection{Thurston matrices}
\begin{itemize}
  \item
  Let {\mapfromto F \Sto \Sto} be a PCF branched covering with post critical set $P$, 
  $\#P > 3$
  \item
  A multicurve $\Gamma = \{\ga_1, \ldots, \ga_n\}$ in $\Sto\sm P$ \emph{is $F$-stable} 
  if for every $j$ and every connected component $\delta$ of $F^{-1}(\gamma_j)$, 
  the simple closed curve $\delta$ is either homotopic to some $\ga_i$ or peripheral in 
  $\Sto\sm P$.
  \item
  The \emph{Thurston Matrix} of $F$ with respect to the $F$-stable multicurve $\Gamma$ 
  is the non negative 
  $n\times n$ matrix $A = A_{i,j}$ given by
  $$
  A_{i,j} = \sum_{\delta} 1/\deg(F:\delta\to \ga_j)
  $$
  where the sum is over all connected components $\delta$ of $F^{-1}(\gamma_j)$ 
  homotopic to $\gamma_i$ relative to $P$, in $\Sto\sm P$.
\end{itemize}  
\subsection{Thurston obstructions}
\begin{itemize}
  \item
  Having only non negative entries, the Thurston matrix $A$ has a positive \emph{leading eigenvalue}, 
  i.e. eigenvalue of maximal modulus.
  \item
  A \emph{Thurston obstruction} to $F$ is an $F$-stable multicurve $\Gamma$ 
  with leading eigenvalue of modulus at least $1$.
\end{itemize}  

\subsection{The fundamental theorem for post-critically finite rational maps.}
\THM[Thurston]
Let {\mapfromto F \Sto \Sto} be a post-critically finite branched covering with post-critical set $P$, 
and hyperbolic orbifold. Then $F$ is Thurston equivalent to a rational map if and only if $F$ has 
no Thurston obstruction.
\ENDTHM

The Orbifold of $F$
 
\begin{itemize}
  \item
  The orbifold $\OO_F$ associated to $F$ is the topological orbifold $(\Sto,\nu)$ 
  with underlying space $\Sto$ and whose weight $\nu(x)$ at $x$ is 
  the least common multiple of the local degree of $F^n$ 
  over all iterated preimages $F^{-n}(x)$ of $x$.
  \item
  The orbifold $\OO_F$ is said to be \emph{hyperbolic} if its Euler characteristic
  $\chi(\OO_F)$ is negative, that is if:
  $$
  \chi(\OO_F) := 2 - \sum_{x\in P} (1 -\frac {1}{\nu(x)}) < 0.
  $$
\end{itemize}

\end{document}